\documentclass[journal]{IEEEtran}
\usepackage{cite}
\usepackage{amsmath}
\usepackage{graphicx}
\usepackage{float}
\usepackage{caption}
\usepackage{color}
\usepackage{CJK}
\usepackage{algorithm}
\usepackage{algorithmic}
\usepackage{indentfirst}
\usepackage{subfigure}
\usepackage{epstopdf}
\usepackage[latin1]{inputenc}
\usepackage{rotating}
\usepackage{colortbl}
\usepackage{amssymb,amsthm,enumerate,graphicx,float}
\usepackage{booktabs}
\usepackage{hyperref}
\ifCLASSINFOpdf
\else
\fi
\newtheorem{theorem}{Theorem}

\newtheorem{lemma}{Lemma}

\newtheorem{assumption}{Assumption}

\newtheorem{remark}{Remark}

\begin{document}
\pagestyle{empty}  
\thispagestyle{empty} 
\onecolumn
{\copyright{2020 IEEE.  Personal use of this material is permitted.  Permission from IEEE must be obtained for all other uses, in any current or future media, including reprinting/republishing this material for advertising or promotional purposes, creating new collective works, for resale or redistribution to servers or lists, or reuse of any copyrighted component of this work in other works.}}
\newpage
\twocolumn
\title{S-DIGing: A Stochastic Gradient Tracking Algorithm for Distributed Optimization}
\author{ Huaqing Li \IEEEmembership{Senior Member}, Lifeng Zheng, Zheng Wang, Yu Yan, Liping Feng, and Jing Guo
\thanks{ The work described in this paper is supported in part by the National Natural Science Foundation of China under Grant 61773321, in part by the Innovation Support Program for Chongqing Overseas Returnees under Grant cx2017043, and in part by the Special Financial Support from Chongqing Postdoctoral Science Foundation under Grant Xm2017100. \emph{(Corresponding author: Yu Yan.)}

 H. Li, L. Zheng, Z. Wang, Y. Yan, and J. Guo are with Chongqing Key Laboratory of Nonlinear Circuits and Intelligent Information Processing, College of Electronic and Information Engineering, Southwest University, Chongqing 400715, PR China (e-mail: huaqingli@swu.edu.cn; zlf\_swu@163.com; swu\_wz@126.com; yanyu\_nice@163.com; poem24@163.com).

 L. Feng is with Department of Computer Science, Xinzhou Teachers University, Shanxi Xinzhou 034000, PR China (e-mail: fenglp@yeah.net).
 }}
\maketitle
\setcounter{page}{1}
\begin{abstract}
In this paper, we study convex optimization problems where agents of a network cooperatively minimize the global objective function which consists of multiple local objective functions.
The intention of this work is to solve large-scale optimization problems where the local objective function is complicated and numerous.
Different from most of the existing works, the local objective function of each agent is presented as the average of finite instantaneous functions.
Integrating the gradient tracking algorithm with stochastic averaging gradient technology, a distributed stochastic gradient tracking (termed as S-DIGing) algorithm is proposed.
At each time instant, only one randomly selected gradient of an instantaneous function is computed and applied to approximate the local batch gradient for each agent.
Based on a novel primal-dual interpretation of the S-DIGing algorithm, it is shown that the S-DIGing algorithm linearly converges to the global optimal solution when step-size do not exceed an explicit upper bound and the instantaneous functions are strongly convex with Lipschitz continuous gradients.
Numerical experiments are presented to demonstrate the practicability of the S-DIGing algorithm and correctness of the theoretical results.
\end{abstract}

\IEEEpeerreviewmaketitle

\section{Introduction}
With the emergence of applications in the fields of wireless sensor network, smart-grid, machine learning and cloud computing, distributed optimization theory and application have received extensive attention, and gradually penetrated into many aspects of scientific research, engineering applications and social life\cite{Sardellitti2014,Jia2018,4447651,6851883,Fan2012,Tsianos2012,qin7839186,yuan7298442,5233809,7429776,8825499}.
Unlike the traditional centralized optimization problem, the concept of distributed optimization problem is that multiple agents in a network work together to minimize the global objective function $\tilde{f}\left( {\tilde{x}} \right)=\sum\nolimits_{i=1}^{m}{{{f}_{i}}\left( {\tilde{x}} \right)}$ in which $f_{i}$ is only known by agent $i$.
Each agent computes the local information of itself and sends the results to its neighbor agents.

In the existing literature, researches on distributed optimization algorithms are mainly based on Newton's method, (sub)gradient descent method and Lagrangian method.
Comparing with the other two methods, the (sub)gradient descent method is comparably simple, where each agent calculates the (sub)gradient of the local objective function and moves the estimation along the negative direction of the (sub)gradient\cite{Nesterov1966}.
Based on (sub)gradient method, Liu et al. \cite{Liu2017} prove that the estimations can converge to a global optimal solution with convergence rate $O(1/k)$ provided that the diminishing step-size satisfies some conditions and the global objective function is strongly convex.
In order to further improve the convergence rate, Nedic et al. \cite{Nedic2017c} combine the distributed inexact gradient method with the gradient tracking technique to introduce the DIGing algorithm.
Employing doubly stochastic mixing matrices and a fixed step-size, the DIGing algorithm can converge at a linear rate as long as the fixed step-size do not exceed some upper bound.
Based on the gradient tracking method proposed in \cite{Nedic2017c}, Maros and Jalden innovatively develop a dual linearly converging method (PANDA).
The advantages of PANDA is that it requires communicating half as many quantities as DIGing per iteration and PANDA's iterates are in general computationally more expensive than those of DIGing \cite{PANDA69}.
More classical results about (sub)gradient method can be found in \cite{Nedic2008,Nedic2009a ,Lobel2011a}.
Unlike the (sub)gradient descent method, algorithms based on the Newton's method usually have faster convergence rates but more expensive computation costs.
This kind of algorithms use the local first-order and second-order partial derivative information to estimate the trait of the global objective function, and obtain the global optimal solution \cite{Mokhtari2017,Wei2013}.
To reduce the high computation cost, the quasi-Newton method with less computational cost is proposed.
The essential idea of the quasi-Newton method is to avoid the defect solving the inverse of the complex Hessian matrix at each time instant.
It employs a positive definite matrix to approximate the inverse of the Hessian matrix, which simplifies the computational complexity\cite{Bolognani2010,Lewis2013}.
The Lagrange multiplier method is mainly used to solve the constrained optimization problem \cite{8708976}.
The basic idea is to transform constrained optimization problems with $m$ variables and $d$ constraints into unconstrained optimization problems with $m+d$ variables by introducing Lagrange multipliers.
A typical example is the decentralized alternating direction method of multipliers (ADMM)\cite{Boyd2010a}, based on which many distributed algorithms are presented\cite{Iutzeler2016,Chang2016}.
Distributed ADMM algorithms show linear convergence rates for strongly convex functions with a fixed step-size, but suffers from heavy computation burden because each agent has to optimize its local objective function at each time instant.
To reduce computation cost, the exact first-order algorithm \cite{Shi2014} is proposed, which is essentially first-order approximations of distributed ADMM.

In distributed settings, all of the aforementioned algorithms require the computationally costly evaluation of the local gradient $\nabla {{f}_{i}}\left( {\tilde{x}} \right)$ when the local objective function ${{f}_{i}}\left( {\tilde{x}} \right)$ is complicated and numerous, such as problems about machine learning, data mining and so on.
This cost can be avoided by stochastic decentralized algorithms that reduce computational cost of iterations by substituting all local gradients with their stochastic approximations\cite{Schmidt2017}.
The DSA algorithm proposed in \cite{Mokhtari2016} combines the EXTRA algorithm and the stochastic gradient technique \cite{Defazio2014} to save computation cost without compromising convergence.
Under strongly convex and Lipschitz continuous gradient conditions, the DSA algorithm can also achieve a linear convergence rate with a fixed step-size.
Inspired by the DSA algorithm, an augmented Lagrange stochastic gradient algorithm is presented to address the distributed optimization problem, which combines the factorization of weighted Laplacian and local unbiased stochastic averaging gradient methods\cite{Introduction2018}.
Based on $\mathsf{\mathcal{A}\mathcal{B}}$ algorithm \cite{Xin2018}, Xin et al. propose a distributed stochastic gradient algorithm, called $\mathsf{\mathcal{S}}\text{-}\mathsf{\mathcal{A}\mathcal{B}}$, where each agent uses an auxiliary variable to asymptotically track the gradient of the global objective function in expectation\cite{Xin2019a}.
Employing row- and column-stochastic weights simultaneously, the $\mathsf{\mathcal{S}}\text{-}\mathsf{\mathcal{A}\mathcal{B}}$ algorithm converges to a neighborhood of the global optimal solution with a linear convergence rate.
Using Hessian information, a linear algorithm, called SUCAG, is introduced in \cite{Wai2019}.
When the initialization point is sufficiently close to the optimal solution, the established convergence rate of the SUCAG algorithm is only dependent on the condition number of the global objective problem, making it strictly faster than the known rate for the SAGA method \cite{Defazio2014}.

In this work, we introduce a novel distributed optimization algorithm by integrating the gradient tacking and stochastic gradient technologies into gradient descent method\cite{Nedic2017c,Defazio2014}.
The S-DIGing algorithm is based on the combination of the DIGing algorithm \cite{Nedic2017c} and the unbiased stochastic gradients introduced in \cite{Defazio2014}.
We propose a new analytical framework, which is completely different from \cite{Nedic2017c}.
Specifically, we iteratively rewrite the S-DIGing algorithm into a general form and let the accumulation estimate be a dual variable to obtain a primal-dual algorithm which is equivalent to the S-DIGing algorithm.
We now summarize the main contributions:
\begin{enumerate}[\quad (1)]
\item The relationship and transformation process between gradient tracking algorithm and primal-dual algorithm are analyzed in detail.

\item Using the unbiased stochastic gradient of local objective function instead of the standard gradient, the S-DIGing algorithm significantly reduces the complexity and the computation cost, which means the S-DIGing algorithm can perform well in large-scale problems.

\item We establish a linear convergence rate for smooth and strongly-convex instantaneous functions when the fixed step-size is positive and do not exceed some explicit upper bound.

\item We cast a novel analytical framework that makes it easier to analyze the conditions of convergence and convergence rates. The relationship between convergence rate and step-size, parameters and network structure is given in this paper.

\item Comparing with the $\mathsf{\mathcal{S}}\text{-}\mathsf{\mathcal{A}\mathcal{B}}$ algorithm \cite{Xin2019a}, the S-DIGing algorithm can exactly converge to the global optimal solution instead of the neighborhood of the global optimal solution.
\end{enumerate}

We now organize the rest of this paper.
Section II formulates the optimization problem, states the network model, provides some necessary assumptions and describe problems of interest.
Section III presents the unbiased stochastic averaging gradient and proposed the S-DIGing algorithm.
The convergence analysis is provided in Section IV.
In order to experimentally verify the results of this paper, we provide simulation results of the S-DIGing algorithm in Section V.
Finally, Section VI summarizes the paper and envision future research.

\emph{Notations:} All vectors throughout the paper default to column vectors.
We write ${{x}^{\rm{T}}}$ and ${{A}^{\rm{T}}}$ to denote the transpose of a vector $x$ and a matrix $A$, respectively.
For a matrix $A$, we denote its $(i,j)$-th element by ${{A}_{ij}}$.
We use $\left\| \cdot  \right\|$ for both vectors and matrices, in the former case, $\left\| \cdot  \right\|$ represents the Euclidean norm whereas in the latter case it indicates the spectral norm.
The notation ${{1}_{n}}$ represents the $n$-dimensional vector of ones and ${{I}}$ represents the identity matrix with proper dimensions.
For a vector $x$, we use ${{\left\| x \right\|}_{G}}$ to denote the $G$-norm of $x$, i.e., ${{\left\| x \right\|}_{G}}=\sqrt{{{x}^{\text{T}}}Gx}$, where $G$ is a positive semi-definite matrix.
We denote by ${{\rho }_{1}(A)}\le {{\rho }_{2}(A)}\le\cdots  \le{{\rho }_{m}(A)}$ the eigenvalues of a real symmetric matrix $A\in {{\mathbb{R}}^{m\times m}}$.
A nonnegative vector is called stochastic if the sum of its elements equals to one.
A nonnegative square matrix is called row- (column-) stochastic if its rows (columns) are stochastic vectors, respectively.
We abbreviate independent and identically distributed to i.i.d..

\section{Problem Definition}

\subsection{Problem Formulation}
Consider a network containing $m$ agents and all agents aim at cooperatively solving the optimization problem as follows:
\begin{flalign}
\underset{\tilde{x}\in {{\mathbb{R}}^{n}}}{\mathop{\min }}\,\tilde{f}\left( {\tilde{x}} \right)=\sum\limits_{i=1}^{m}{{{{{f}}}_{i}}\left( {\tilde{x}} \right)}=\sum\limits_{i=1}^{m}{\frac{1}{{{q}_{i}}}\sum\limits_{h=1}^{{{q}_{i}}}{f_{i}^{h}\left( {\tilde{x}} \right)}}
\end{flalign}
where agent $i$ possesses exclusive knowledge of its local objective function $f_i$.
The goal is to seek the global optimal solution ${{\tilde{x}}^{*}}\in {{\mathbb{R}}^{n}}$ to (1) via only local computations and communication among agents.

Let $m$ agents be connected over an undirected graph, $\mathsf{\mathcal{G}}=\left( \mathsf{\mathcal{V}},\mathsf{\mathcal{E}},\mathsf{\mathcal{W}} \right)$, where $\mathsf{\mathcal{V}}=\left\{ 1,\ldots ,m \right\}$ is the set of agents, $\mathsf{\mathcal{E}}\subseteq \mathsf{\mathcal{V}}\times \mathsf{\mathcal{V}}$ is the collection of edges, and $\mathsf{\mathcal{W}}=\left[ {{w}_{ij}} \right]\in {{\mathbb{R}}^{m\times m}}$ indicates the weighted adjacency matrix where the weight $w_{ij}$ associated with edge $(i,j)$ satisfies: $w_{ij}>0$ if $(i,j)\in \mathsf{\mathcal{E}}$; and $w_{ij}=0$, otherwise.
Assume that $(i,i)\in \mathsf{\mathcal{E}}$ and set ${{w}_{ii}}=1-\sum\nolimits_{j=1,j\ne i}^{m}{{{w}_{ij}}}> 0$.
Two agents $i$ and $j$ can only communicate directly with each other if the edge $\left( i,j \right)\in \mathsf{\mathcal{E}}$.

Then, we equivalently reformulate the optimization problem (1) as follows:
\begin{flalign}
\begin{split}
&  \underset{{x}\in {{\mathbb{R}}^{mn}}}{\mathop{\min }}\,f\left( x \right)=\sum\limits_{i=1}^{m}{{{f}_{i}}\left( {{x}^{i}} \right)}=\sum\limits_{i=1}^{m}{\frac{1}{{{q}_{i}}}\sum\limits_{h=1}^{{{q}_{i}}}{f_{i}^{h}\left( {{x^{i}}} \right)}} \\
 &\;\;\;\text{s}\text{.t}\text{. }    \;\,  {{x}^{i}}={{x}^{j}}, \forall (i,j)\in \mathsf{\mathcal{E}}
\end{split}
\end{flalign}
where $x = {\left[ {({x^1})^{{\mathop{\rm T}\nolimits}} \text{,}{\kern 1pt}  \cdots \text{,}{\kern 1pt} {\kern 1pt} {\kern 1pt} ({x^m}) ^{\mathop{\rm T}\nolimits} } \right]^{\rm T}} \in {\mathbb{R}^{mn}}$.
Therefore, the global optimal solution, ${{x}^{*}}\in {{\mathbb{R}}^{mn}}$, to problem (2) is equal to ${{1}_{m}}\otimes {{\tilde{x}}^{*}}$.
\begin{assumption}
The graph, $\mathsf{\mathcal{G}}$, is undirected and connected.
\end{assumption}
\begin{assumption}
Each instantaneous function $f_{i}^{h}$ is strongly convex and has Lipschitz continuous gradient, i.e., for all $a,b\in {{\mathbb{R}}^{n}}$, we have
\begin{flalign}
{{\left( \nabla f_{i}^{h}\left( a \right)-\nabla f_{i}^{h}\left( b \right) \right)}^{\rm{T}}}\left( a-b \right)\ge \mu {{\left\| a-b \right\|}^{2}}
\end{flalign}
and
\begin{flalign}
\left\| \nabla f_{i}^{h}\left( a \right)-\nabla f_{i}^{h}\left( b \right) \right\|\le {{L}_{f}}\left\| a-b \right\|
\end{flalign}
where ${{L}_{f}}>\mu>0$.
\end{assumption}
\subsection{Problems of Interest}
Problems of particular interest are those involving lots of miscellaneous local objective function in which exact calculation of the gradients are impossible or computationally intractable. Here we provide two such examples:
\subsubsection{Distributed Logistic Regression}
The purpose of distributed logistic regression is to predict the probability that the dependent variable $l_{i,h}$ is $+1$.
The probability can be computed as $P\left( \left. l_{i,h}=1 \right|c_{i,h} \right)= 1/(1+\exp ( -{{l}_{i,h}}c_{i,h}^{\text{T}}\tilde{x} ))$.
It follows from this model that the regularized maximum log likelihood estimate of the classifier $\tilde{x}$ given the training samples $\left( {{c}_{i,h}},{{l}_{i,h}} \right)$ for $h=1,\ldots ,{{q}_{i}}$ and $i=1,\ldots ,m$, is the optimal solution of the optimization problem
\begin{flalign*}
{{\tilde{x}}^{*}}=\arg \underset{\tilde{x}\in {{\mathbb{R}}^{n}}}{\mathop{\min }}\left(\frac{\lambda }{2}{{\left\| {\tilde{x}} \right\|}^{2}}+\sum\limits_{i=1}^{m}{\sum\limits_{h=1}^{{{q}_{i}}}{\log \left( 1+\exp \left( -{{l}_{i,h}}c_{i,h}^{\text{T}}\tilde{x} \right) \right)}}\right)
\end{flalign*}
where the regularization term $\frac{\lambda }{2}{{\left\| {\tilde{x}} \right\|}^{2}}$ is added to reduce over-fitting to the training set.
\subsubsection{Energy-Based Source Localization}
Estimating the location of an acoustic source is an important problem in both environment and military \cite{Chen2002}.
In this problem, an acoustic source is positioned at an unknown location, ${\tilde{x}}$, in a sensor field.
We use an isotropic energy propagation model for the $h$-th received signal strength measurement at each agent $i\in \mathsf{\mathcal{V}}$: ${{c}_{i,h}}=\frac{a}{{{\left\| \tilde{x}-{{r}_{i}} \right\|}^{\theta }}}+{{\upsilon }_{i,h}}$, $h=1,\ldots ,{{q}_{i}}$, where $a>0$ is a constant and ${{r}_{i}}\in {{\mathbb{R}}^{n}}$ is the location of agent $i$ relative to a fixed reference point.
The exponent $\theta \ge 1$ describes an attenuation characteristic of the medium through which the acoustic signal
propagates, and ${{\upsilon }_{i,h}}$ are i.i.d. samples of a zero-mean Gaussian noise process with variance ${{\sigma }^{2}}$.
A maximum likelihood estimate for the source's location is found by solving
\begin{flalign*}
{{{\tilde{x}}}^{*}}=\arg \underset{\tilde{x}\in {{\mathbb{R}}^{n}}}{\mathop{\min }}\,\sum\limits_{i=1}^{m}{\frac{1}{{{q}_{i}}}\sum\limits_{h=1}^{{{q}_{i}}}{\left( {{c}_{i,h}}-\frac{a}{{{\left\| \tilde{x}-{{r}_{i}} \right\|}^{\theta }}} \right)}}^2
\end{flalign*}
\section{Algorithm Development}
\subsection{Review of Gradient Tracking Algorithm}
We now review the gradient tracking (DIGing) algorithm \cite{Nedic2017c} as follows:
\begin{flalign}
  & x_{k+1}^{i}=\sum\limits_{j=1}^{m}{{{w}_{ij}}x_{k}^{j}}-\alpha y_{k}^{i} \tag{5a}\\
 & y_{k+1}^{i}=\sum\limits_{j=1}^{m}{{{w}_{ij}}y_{k}^{j}}+\nabla {{f}_{i}}\left( {{x}_{k+1}^i} \right)-\nabla {{f}_{i}}\left( {{x}_{k}^i} \right) \tag{5b}
\end{flalign}
where $\alpha$ is the fixed step-size.
At each time instant $k$, agent $i\in \mathsf{\mathcal{V}}$ maintains two variables, $x_{k}^{i},y_{k}^{i}\in {{\mathbb{R}}^{n}}$, initialized with $x_0^i=0$ and $y_{0}^{i}=\nabla {{f}_{i}}\left( x_{0}^{i} \right)$.
The update of $x_k^i$ of agent $i$ is classically gradient descent, and the descent direction is given by an estimate of the global gradient, $y_{k}^{i}$, instead of the local gradient, $\nabla {{f}_{i}}\left( x_{k}^{i} \right)$.
The update of $y_k^i$ of agent $i$ tracks the global gradient and is based on weight matrix.

For the convenience of analysis, the variables ${{x}_{k}}$ and ${{y}_{k}}$ collect the local variables ${{x}_{k}^i}$ and ${{y}_{k}^i}$ in a vector form, respectively, i.e., ${x_k} = {[{(x_k^1)^{\rm{T}}}, \ldots ,{(x_k^m)^{\rm{T}}}]^{\rm{T}}}$ and ${y_k} = {[{(y_k^1)^{\rm{T}}}, \ldots ,{(y_k^m)^{\rm{T}}}]^{\rm{T}}}$.
Defining $\nabla F\left( {{x}_{k}} \right)={{\left[ \nabla f_{1}^{\rm{T}}\left( x_{k}^{1} \right),\cdots ,\nabla f_{m}^{\rm{T}}\left( x_{k}^{m} \right) \right]}^{\rm{T}}}$.
Algorithm (5) can be equivalently rewritten as:
\begin{flalign}
  & {{x}_{k+1}}=W{{x}_{k}}-\alpha {{y}_{k}} \tag{6a}\label{8a}\\
 & {{y}_{k+1}}=W{{y}_{k}}+\nabla F\left( {{x}_{k+1}} \right)-\nabla F\left( {{x}_{k}} \right) \tag{6b}\label{8b}
\end{flalign}
where $W=\left( \mathsf{\mathcal{W}}\otimes I \right)$, ${{x}_{0}}=0$ and ${{y}_{0}}=\nabla F\left( {{x}_{0}} \right)$.
\setcounter{equation}{6}
\subsection{Unbiased Stochastic Averaging Gradient}
Recall that the definitions of the local function $f_{i}(x^i)$ and the instantaneous functions $f_{i}^{h}(x^i)$ available at agent $i$, the implementation of gradient tracking algorithm requires that each agent $i$ computes the full gradient of its instantaneous functions $f_i^{h}$ at $x_{k}^{i}$ as
\begin{flalign}
\nabla {{f}_{i}}\left( x_{k}^{i} \right)=\frac{1}{{{q}_{i}}}\sum\limits_{h=1}^{{{q}_{i}}}{\nabla f_{i}^{h}\left( x_{k}^{i} \right)}
\end{flalign}
This is computationally expensive especially when the number of instantaneous functions $q_{i}$ is large.
To solve this issue, we utilize a localized SAGA technology inspired by \cite{Mokhtari2016}.
An unbiased stochastic averaging gradient is employed to substitute the costly full gradient computation.
It approximates the gradient $\nabla f_{i}(x_{k}^{i})$ of agent $i$ at time instant $k$ by randomly choosing
one of the instantaneous functions gradients $\nabla f_{i}^{h}(x_{k}^{i})$, $h\in \{1,\dots,q_{i}\}$.
Let $t_{k+1}^{i}\in \{1,\ldots ,{{q}_{i}}\}$ denote a function index that we choose at time instant $k$ on agent $i$ uniformly at random.
For agent $i$, the update $z_{k+2}^{i,h}$ can be presented as follows:
\begin{flalign*}
\left\{ \begin{matrix}
   z_{k+2}^{i,h}=x_{k+1}^{i},  \;\; \quad  \text{if} \quad h=t_{k+1}^{i}  \\
   z_{k+2}^{i,h}=z_{k+1}^{i,h},  \;\;\quad \text{if} \quad h\ne t_{k+1}^{i}  \\
\end{matrix} \right.
\end{flalign*}
where $z_{1}^{i,h} = z_{0}^{i,h} = x_{0}^{i}=0$, $\forall h=1,\ldots ,{{q}_{i}}$.
Then, we define the stochastic averaging gradient at agent $i$ as
\begin{flalign}\label{g_k^i}
g_{k+1}^{i}=\nabla f_{i}^{t_{k+1}^{i}}( x_{k+1}^{i} )-\nabla f_{i}^{t_{k+1}^{i}}( z_{k+1}^{i,t_{k+1}^{i}} )+\frac{1}{{{q}_{i}}}\sum\limits_{h=1}^{{{q}_{i}}}{\nabla f_{i}^{h}( z_{k+1}^{i,h} )}
\end{flalign}
Letting ${{\mathsf{\mathcal{F}}}_{k}}$ measure the history of the system up until time instant $k$, we have $\mathbb{E}\left[ \left. g_{k}^{i } \right|{{\mathsf{\mathcal{F}}}_{k}} \right]=\nabla {{f}_{i}}\left( x_{k}^{i} \right)$, $\forall i\in \mathsf{\mathcal{V}}$, which means that the stochastic averaging gradient is unbiased.
\begin{remark}
Consider Eq. \eqref{g_k^i}. Computation of local averaging gradient $g_{k}^{i}$ is costly because it requires evaluation of a sum $\sum\nolimits_{h=1}^{{{q}_{i}}}{\nabla f_{i}^{h}\left( z_{k}^{i,h} \right)}$ at each time instant \cite{Mokhtari2016}. This cost can be avoided by
updating the sum at each time instant with the recursive formula
\begin{flalign}\label{y}
\begin{split}
  &    \sum\limits_{h=1}^{{{q}_{i}}}{\nabla f_{i}^{h}( z_{k}^{i,h} )} \\
 =&\sum\limits_{h=1}^{{{q}_{i}}}{\nabla f_{i}^{h}( z_{k-1}^{i,h} )}+\nabla f_{i}^{t_{k}^{i}}( x_{k-1}^{i} )-\nabla f_{i}^{t_{k}^{i}}( z_{k-1}^{i,t_{k-1}^{h}} )
 \end{split}
\end{flalign}
Using the update in \eqref{y}, we can update the sum $\sum\nolimits_{h=1}^{{{q}_{i}}}{\nabla f_{i}^{h}( z_{k}^{i,h} )}$ required for (8) in a computationally efficient manner.

\end{remark}
\subsection{Stochastic Gradient Tracking Algorithm}
To solve problem (2) in a computation-efficient way, we propose a stochastic gradient tracking (S-DIGing) algorithm as shown in Algorithm 1 by combining the gradient tracking algorithm with unbiased stochastic averaging gradients technology.
\begin{algorithm}[!ht]
	\caption{: S-DIGing}
	\begin{algorithmic}[1]
		\STATE \textbf{Initialization:} Each agent $i\in \mathcal{V}$ initializes with $z_{1}^{i,h} = z_{0}^{i,h} = x_{0}^{i}=0$, $h=1,\ldots ,{{q}_{i}}$, and $g_  {0}^{i}=y_{0}^{i}=\nabla {{f}_{i}}\left( {{x}_{0}^{i}} \right)$.
        \STATE \textbf{Set} $k=0$.
		\STATE \textbf{For} $i=1$ to $m$ \textbf{do}
		\STATE  \quad Update variable $x_{k+1}^{i}$ as
\begin{flalign*}
  x_{k+1}^{i}=\sum\limits_{j=1}^{m}{{{w}_{ij}}x_{k}^{j}}-\alpha y_{k}^{i}
\end{flalign*}
		\STATE  \quad Choose $t_{k+1}^{i}$ uniformly at random from set $\{1,\ldots ,{{q}_{i}}\}$
		\STATE  \quad Take $z_{k+2}^{i,t_{k+1}^{i}}=x_{k+1}^{i}$ and store $\nabla f_{i}^{t_{k+1}^{i}}( z_{k+2}^{i,t_{k+1}^{i}} )=\nabla f_{i}^{t_{k+1}^{i}}\left( x_{k+1}^{i} \right)$ in $i_{k+1}^{h}$ gradient table position. All other entries in the table remain unchanged, i.e., $z_{k+2}^{i,h}=z_{k+1}^{i,h}$ for all $h\ne t_{k+1}^{i}$
		\STATE  \quad Compute and store $g_{k+1}^{i}$ as (8)
  		\STATE  \quad Update variable $y_{k+1}^{i}$ as
\begin{flalign*}
 y_{k+1}^{i}=\sum\limits_{j=1}^{m}{{{w}_{ij}}y_{k}^{j}}+g_{k+1}^{i}-g_{k}^{i} 
\end{flalign*}
        \STATE \textbf{Set} $k\to k+1$ and go to Step 3 until a certain stopping criterion is satisfied, e.g., the maximum number of iterations.
	\end{algorithmic}
\end{algorithm}

Recall that the definitions of $x_{k}$ and $y_{k}$, algorithm 1 can be equivalently rewritten as the following matrix-vector form:
\begin{align}
  & {{x}_{k+1}}=W{{x}_{k}}-\alpha {{y}_{k}} \tag{10a}\\
 & {{y}_{k+1}}=W{{y}_{k}}+{{g}_{k+1}}-{{g}_{k}} \tag{10b}
\end{align}
where ${{x}_{0}}=0$, ${{y}_{0}}={{g}_{0}}$ and ${{g}_{k}}={{\left[ {{\left( g_{k}^{1} \right)}^{\text{T}}},\ldots ,{{\left( g_{k}^{m} \right)}^{\text{T}}} \right]}^{\text{T}}}$.
\setcounter{equation}{10}
\subsection{Primal-Dual Interpretation of S-DIGing}
Considering problem (2), the constraints ${{x}_{i}}={{x}_{j}}$,$\forall \left( i,j \right)\in \mathsf{\mathcal{E}}$, can be equivalently written as a matrix-vector form of $( {\tilde{L}}\otimes I )x=0$, where ${\tilde{L}}=\left[ {{l}_{ij}} \right]\in {{\mathbb{R}}^{m\times m}}$ is a matrix satisfying  ${{l}_{ij}}=-{{w}_{ij}}$ for $i\ne j$ and ${{l}_{ij}}=1-{{w}_{ij}}$ for $i=j$.
Letting $L = ({\tilde{L}}\otimes I)$, the augmented Lagrange function is constructed as follows:
\begin{flalign*}
\mathsf{\mathcal{L}}\left( x,\lambda  \right)=f\left( x \right)+\frac{1}{\alpha }{{\lambda }^{\text{T}}}Lx+\frac{1}{2\alpha }{{x}^{\text{T}}}\left( I-{{W}^{2}} \right)x
\end{flalign*}
where $\lambda$ is the Lagrange multiplier and $\alpha>0$.
Thus, the constrained optimization problem (2) can be transformed into a saddle point finding problem.
Considering the partial derivatives ${{\partial }_{x}}\mathsf{\mathcal{L}}\left( x,\lambda \right)=\nabla f\left( x \right)+(1/\alpha )L\lambda +(1/\alpha )(I-{{W}^{2}})x$ and ${{\partial }_{\lambda }}\mathsf{\mathcal{L}}\left( x,\lambda  \right)=(1/\alpha )Lx$, we describe the classical primal-dual algorithm solving problem (2) as follows:
\begin{align*}
  & {{x}_{k+1}}={{x}_{k}}-\alpha \left( \nabla F\left( {{x}_{k}} \right)+\frac{1}{\alpha }L{{\lambda }_{k}}+\frac{1}{\alpha }\left( I-{{W}^{2}} \right){{x}_{k}} \right) \\
 & {{\lambda }_{k+1}}={{\lambda }_{k}}+\alpha \left( \frac{1}{\alpha }L{{x}_{k+1}} \right)
\end{align*}
which is equal to
\begin{align}
  & {{x}_{k+1}}={{W}^{2}}{{x}_{k}}-\alpha \nabla F\left( {{x}_{k}} \right)-L\lambda_{k}  \tag{11a}\\
 & {{\lambda }_{k+1}}={{\lambda }_{k}}+L{{x}_{k+1}} \tag{11b}
\end{align}
where $x_{0}=0$, $\lambda_{0} = 0$ and $\alpha$ is the step-size.
\setcounter{equation}{11}
It can be found that the $x_{k}$-update at each agent is essentially gradient descent and the $\lambda_{k}$-update at each agent is gradient ascent.

Then, we establish a relationship between the S-DIGing algorithm (10) and the primal-dual algorithm (11).
Rewriting algorithm (10) recursively, we get, $\forall k\ge 1$,
\begin{flalign}
{{x}_{k+1}}&=2W{{x}_{k}}-{{W}^{2}}{{x}_{k-1}}-\alpha \left( g_{k}-g_{k-1} \right) \label{E13}
\end{flalign}
Since $x_{0}=0$ and $x_{1}=Wx_{0}-\alpha g_{0}$, subtracting $x_{k}$ from both sides of \eqref{E13}, we obtain, for all $k\ge1$,
\begin{flalign}\label{E15}
{{x}_{k+1}}-{{x}_{k}}={(2{W}-I)}{{x}_{k}}-{{W}^{2}}{{x}_{k-1}}-\alpha \left( g_{k}-g_{k-1} \right)
\end{flalign}
Adding the first update ${{x}_{1}}=W{{x}_{0}}-\alpha {{g}_{0}}$ to the subsequent updates following the formulas of $\left( {{x}_{2}}-{{x}_{1}} \right)$, $\left( {{x}_{3}}-{{x}_{2}} \right)$, $\dots$, $\left( {{x}_{k+1}}-{{x}_{k}} \right)$ given by \eqref{E15} and applying telescopic cancellation, we have
\begin{flalign}\label{E16}
{{x}_{k+1}}={{W}^{2}}{{x}_{k}}-\alpha g_{k}-\sum\limits_{s=0}^{k}{\left( I-W \right)^2{{x}_{s}}}
\end{flalign}
Letting ${{U}}={{ I-W }} \ge 0$, we rewrite \eqref{E16} as
\begin{flalign}\label{E17}
{{x}_{k+1}}={{W}^{2}}{{x}_{k}}-\alpha g_{k}-U\sum\limits_{s=0}^{k}{U{{x}_{s}}}
\end{flalign}
Defining ${{\lambda }_{k}}=\sum\nolimits_{s=0}^{k}{U{{x}_{s}}}$, Eq. \eqref{E17} is equal to
\begin{flalign}
  & {{x}_{k+1}}={{W}^{2}}{{x}_{k}}-\alpha g_{k}-L{{\lambda }_{k}} \tag{16a}\\
 & {{\lambda }_{k+1}}={{\lambda }_{k}}+L{{x}_{k+1}} \tag{16b}
\end{flalign}
where $L=U=I-W$, $x_{0}=0$ and $\lambda_{0}=0$.
It can be found that the algorithm (16) is completely equivalent to the stochastic gradient form of the primal-dual algorithm with $x_{0}=0$ and $\lambda_{0}=0$.
\setcounter{equation}{16}
\section{Convergence Analysis}
In this section, we show the convergence analysis process of Algorithm 1.
Defining ${{V}_{1}}\left( {{x}_{k}},{{\lambda }_{k}} \right)=\left\| {{x}_{k}}-{{x}^{*}} \right\|_{{{W}^{2}}}^{2}+{{\left\| {{\lambda }_{k}}-{{\lambda }^{*}} \right\|}^{2}}$, for all $k\ge0$, we have
\begin{flalign}\label{E19}
\begin{split}
  & {{V}_{1}}\left( {{x}_{k+1}},{{\lambda }_{k+1}} \right)-{{V}_{1}}\left( {{x}_{k}},{{\lambda }_{k}} \right) \\
 =&-\left\| {{x}_{k+1}}-{{x}_{k}} \right\|_{{{W}^{2}}}^{2}-{{\left\| {{\lambda }_{k+1}}-{{\lambda }_{k}} \right\|}^{2}}+2{{\left( {{x}_{k+1}}-{{x}^{*}} \right)}^{\rm{T}}}\\
 &\times {{W}^{2}}\left( {{x}_{k+1}}-{{x}_{k}} \right)+2{{\left( {{\lambda }_{k+1}}-{{\lambda }^{*}} \right)}^{\rm{T}}}\left( {{\lambda }_{k+1}}-{{\lambda }_{k}} \right)
 \end{split}
\end{flalign}
Next, we state an upper bound for the third term of the right hand side of \eqref{E19}.
\begin{lemma}
Consider the algorithm in (16) and let Assumptions 1 and 2 hold. For all $k \ge 0$, we have
\begin{flalign*}
  & 2{{\left( {{x}_{k+1}}-{{x}^{*}} \right)}^{\rm{T}}}{{W}^{2}}\left( {{x}_{k+1}}-{{x}_{k}} \right) \\
 \le & \left\| {{x}_{k+1}}-{{x}^{*}} \right\|_{2\left( {{W}^{2}}-I+{{L}^{2}} \right)-\alpha \left( 2\mu -\phi  \right)I}^{2} \\
 &   -2{{\left( {{\lambda }_{k+1}}-{{\lambda }_{k}} \right)}^{\rm{T}}}\left( {{\lambda }_{k+1}}-{{\lambda }^{*}} \right) \\
 &    +\alpha ( \eta +\frac{L_{f}^{2}}{\phi } ){{\left\| {{x}_{k+1}}-{{x}_{k}} \right\|}^{2}}+\frac{\alpha }{\eta }{{\left\| {{g}_{k}}-\nabla F({{x}_{k}}) \right\|}^{2}} \\
 &    -2\alpha {{\left( {{x}_{k}}-{{x}^{*}} \right)}^{\rm{T}}}\left( {{g}_{k}}-\nabla F({{x}_{k}}) \right)
\end{flalign*}
where $\eta >0$ and $0<\phi<2\mu$.
\end{lemma}
\begin{proof}
Recalling ${{x}_{k+1}}={{W}^{2}}{{x}_{k}}-\alpha {{g}_{k}}-L{{\lambda }_{k}}$, we have
\begin{flalign}\label{L1L1}
\begin{split}
&\alpha {{g}_{k}}\\
  =&{{W}^{2}}{{x}_{k}}-{{x}_{k+1}}-L{{\lambda }_{k+1}}+{{L}^{2}}{{x}_{k+1}} \\
  =&{{W}^{2}}\left( {{x}_{k}}-{{x}_{k+1}} \right)+\left( {{W}^{2}}-I+{{L}^{2}} \right){{x}_{k+1}}-L{{\lambda }_{k+1}}
\end{split}
\end{flalign}
By subtracting $\alpha \nabla F\left( {{x}^{*}} \right)$ from the both sides of \eqref{L1L1} and considering the fact $\alpha \nabla F\left( {{x}^{*}} \right)=-L{{\lambda }^{*}}$, one has
\begin{flalign}\label{L1L2}
  &    \alpha \left( {{g}_{k}}-\nabla F\left( {{x}^{*}} \right) \right) \notag\\
  =&{{W}^{2}}\left( {{x}_{k}}-{{x}_{k+1}} \right)+\left( {{W}^{2}}-I+{{L}^{2}} \right)\left( {{x}_{k+1}}-{{x}^{*}} \right) \tag{19}\\
  &-L\left( {{\lambda }_{k+1}}-{{\lambda }^{*}} \right) \notag
\end{flalign}
Multiplying both sides of \eqref{L1L2} by $2{{({{x}_{k+1}}-{{x}^{*}})}^{\text{T}}}$, we obtain
\begin{flalign}\label{Eqq}
  & 2{{\left( {{x}_{k+1}}-{{x}^{*}} \right)}^{{\rm{T}}}}{{W}^{2}}\left( {{x}_{k+1}}-{{x}_{k}} \right) \notag\\
 =&\left\| {{x}_{k+1}}-{{x}^{*}} \right\|_{2\left( {{W}^{2}}-I+{{L}^{2}} \right)}^{2}-2{{\left( {{x}_{k+1}}-{{x}^{*}} \right)}^{\rm{T}}}L\left( {{\lambda }_{k+1}}-{{\lambda }^{*}} \right) \notag\\
 &    -2\alpha {{\left( {{x}_{k+1}}-{{x}^{*}} \right)}^{\rm{T}}}\left( {{g}_{k}}-\nabla F\left( {{x}^{*}} \right) \right) \tag{20}
\end{flalign}
Next, we establish an upper bound of the third term of the right hand side of Eq. \eqref{Eqq}. Using the  basic inequality: $2{{a}^{\text{T}}}b\le (1/\phi){{\left\| a \right\|}^{2}}+\phi {{\left\| b \right\|}^{2}}$, $\forall a\in {{\mathbb{R}}^{n}},b\in {{\mathbb{R}}^{n}}$, $\phi >0$, we get
\begin{flalign}
  & {{\left( {{x}_{k+1}}-{{x}^{*}} \right)}^{\rm{T}}}\left( {{g}_{k}}-\nabla F({{x}^{*}}) \right) \notag\\
 =&{{\left( {{x}_{k+1}}-{{x}_{k}} \right)}^{\rm{T}}}\left( {{g}_{k}}-\nabla F({{x}_{k}}) \right)+{{\left( {{x}_{k}}-{{x}^{*}} \right)}^{\rm{T}}}\left( {{g}_{k}}-\nabla F({{x}_{k}}) \right) \notag\\
 & +{{\left( {{x}_{k+1}}-{{x}^{*}} \right)}^{\rm{T}}}\left( \nabla F({{x}_{k}})-\nabla F({{x}_{k+1}}) \right) \notag\\
 & +{{\left( {{x}_{k+1}}-{{x}^{*}} \right)}^{\rm{T}}}\left( \nabla F({{x}_{k+1}})-\nabla F({{x}^{*}}) \right) \notag\\
 \ge & -\frac{\eta }{2}{{\left\| {{x}_{k+1}}-{{x}_{k}} \right\|}^{2}}-\frac{1}{2\eta }{{\left\| {{g}_{k}}-\nabla F({{x}_{k}}) \right\|}^{2}} \tag{21}\\
 & +{{\left( {{x}_{k}}-{{x}^{*}} \right)}^{\rm{T}}}\left( {{g}_{k}}-\nabla F({{x}_{k}}) \right)-\frac{\phi }{2}{{\left\| {{x}_{k+1}}-{{x}^{*}} \right\|}^{2}} \notag\\
 & -\frac{L_{f}^{2}}{2\phi }{{\left\| {{x}_{k+1}}-{{x}_{k}} \right\|}^{2}}+\mu {{\left\| {{x}_{k+1}}-{{x}^{*}} \right\|}^{2}}\notag
\end{flalign}
where $\eta >0$ and $0 < \phi  < 2\mu$.
The proof is completed.
\end{proof}

Combining Lemma 1 and \eqref{E19}, it follows
\begin{flalign}\label{E23}
  & {{V}_{1}}\left( {{x}_{k+1}},{{\lambda }_{k+1}} \right)-{{V}_{1}}\left( {{x}_{k}},{{\lambda }_{k}} \right) \notag\\
\le &-\left\| {{x}_{k+1}}-{{x}_{k}} \right\|_{{{W}^{2}}-\alpha ( \eta +\frac{L_{f}^{2}}{\phi } )I}^{2} \notag\\
&+\left\| {{x}_{k+1}}-{{x}^{*}} \right\|_{2\left( {{W}^{2}}-I+{{L}^{2}} \right)-{{L}^{2}}-\alpha \left( 2\mu -\phi  \right)I}^{2} \tag{22}\\
 &    +\frac{\alpha }{\eta }{{\left\| {{g}_{k}}-\nabla F({{x}_{k}}) \right\|}^{2}}-2\alpha {{\left( {{x}_{k}}-{{x}^{*}} \right)}^{\rm{T}}}\left( {{g}_{k}}-\nabla F({{x}_{k}}) \right) \notag
\end{flalign}
In order to process the right hand of \eqref{E23}, we introduce two important supporting lemmas.
\setcounter{equation}{23}
\begin{lemma}
[\cite{Mokhtari2016}] If Assumptions 1 and 2 hold, the squared norm of the difference between the stochastic averaging gradient $g_{k}$ and the optimal gradient $\nabla F\left( {{x}^{*}}\right)$ in expectation is bounded by
\begin{flalign*}
  & \mathbb{E}\left[ \left. {{\left\| {{g}_{k}}-\nabla F\left( {{x}^{*}} \right) \right\|}^{2}} \right|{{\mathsf{\mathcal{F}}}_{k}} \right] \\
 \le & 4{{L}_{f}}{{p}_{k}}+2\left( 2L_{f}-\mu  \right)\left( f\left( {{x}_{k}} \right)-f\left( {{x}^{*}} \right)-\left\langle \nabla F\left( {{x}^{*}} \right),{{x}_{k}}-{{x}^{*}} \right\rangle  \right)
\end{flalign*}
where
\begin{flalign*}
{{p}_{k}}=\sum\limits_{i=1}^{m}{\frac{1}{{{q}_{i}}}\sum\limits_{h=1}^{{{q}_{i}}}{\left( f_{i}^{h}( y_{k}^{i,h} )-f_{i}^{h}\left( {{{\tilde{x}}}^{*}} \right)-\left\langle \nabla f_{i}^{h}\left( {{{\tilde{x}}}^{*}} \right),y_{k}^{i,h}-{{{\tilde{x}}}^{*}} \right\rangle  \right)}}
\end{flalign*}
\end{lemma}
\begin{lemma}
[\cite{Mokhtari2016}] Define $q_{\min}$ and $q_{\max}$ as the smallest and largest values for the number of instantaneous functions at an agent, respectively, i.e., ${{q}_{\min }}=\underset{i\in \mathsf{\mathcal{V}}}{\mathop{\min }}\,\{{{q}_{i}}\}$ and ${{q}_{\max }}=\underset{i\in \mathsf{\mathcal{V}}}{\mathop{\max }}\,\{{{q}_{i}}\}$.
If Assumptions 1 and 2 hold, for all $k\ge0$, the sequence $p_{k}$ satisfies
\begin{flalign*}
 \mathbb{E}\left[ \left. {{p}_{k+1}} \right|{{\mathsf{\mathcal{F}}}_{k}} \right]-{{p}_{k}} \le -\frac{1}{{{q}_{\max }}}{{p}_{k}}+\frac{1}{{{q}_{\min }}}\frac{{{L}_{f}}}{2}{{\left\| {{x}_{k}}-{{x}^{*}} \right\|}^{2}}
\end{flalign*}
\end{lemma}
\begin{lemma}
Suppose that Assumptions 1 and 2 hold. For all $k \ge 0$, we have
\begin{flalign*}
  & \mathbb{E}\left[ \left. {{V}_{1}}\left( {{x}_{k+1}},{{\lambda }_{k+1}} \right)+\left\| {{x}_{k+1}}-{{x}^{*}} \right\|_{\gamma Q}^{2}+c{{p}_{k+1}} \right|{{F}_{k}} \right] \\
 &   -{{V}_{1}}\left( {{x}_{k}},{{\lambda }_{k}} \right)-\left\| {{x}_{k}}-{{x}^{*}} \right\|_{\gamma Q}^{2}-c{{p}_{k}}\\
 \le & -\mathbb{E}\left[ \left. \left\| {{x}_{k+1}}-{{x}_{k}} \right\|_{{{W}^{2}}-\alpha \left( \eta +\frac{L_{f}^{2}}{\phi } \right)I}^{2} \right|{{\mathsf{\mathcal{F}}}_{k}} \right] \\
 &    +\mathbb{E}\left[ \left. \left\| {{x}_{k+1}}-{{x}^{*}} \right\|_{2\left( {{W}^{2}}-I+{{L}^{2}} \right)-{{L}^{2}}-\alpha \left( 2\mu -\phi  \right)I+\gamma Q}^{2} \right|{{\mathsf{\mathcal{F}}}_{k}} \right] \\
 &     +\left\| {{x}_{k}}-{{x}^{*}} \right\|_{\frac{\alpha }{\eta }\left( 2{{L}_{f}}-\mu  \right){{L}_{f}}I+\frac{c{{L}_{f}}}{2{{q}_{\min }}}I-\gamma Q}^{2}+( \frac{4\alpha }{\eta }{{L}_{f}}-\frac{c}{{{q}_{\max }}} ){{p}_{k}}\\
 \triangleq & -{{\Delta }_{k+1}}
\end{flalign*}
where $Q=\left( I+3W \right)\left( I-W \right)+\alpha \left( 2\mu -\phi  \right)I>0$, $0<\gamma<1$, $\eta>0$, $c>0$ and $0<\phi<2\mu$.
\end{lemma}
\begin{proof}
\setcounter{equation}{22}
Taking the full conditional expectation of Eq. \eqref{E23} and using Lemma 2 to deal with the upper bound of $\mathbb{E}\left[ \left. {{\left\| {{g}_{k}}-\nabla F\left( {{x}^{*}} \right) \right\|}^{2}} \right|{{\mathsf{\mathcal{F}}}_{k}} \right]$,  it can be verified that
\begin{flalign}\label{E24}
\begin{split}
  & \mathbb{E}\left[ \left. {{V}_{1}}\left( {{x}_{k+1}},{{\lambda }_{k+1}} \right) \right|{{\mathsf{\mathcal{F}}}_{k}} \right]-{{V}_{1}}\left( {{x}_{k}},{{\lambda }_{k}} \right) \\
 \le & -\mathbb{E}\left[ \left. \left\| {{x}_{k+1}}-{{x}_{k}} \right\|_{{{W}^{2}}-\alpha \left( \eta +\frac{L_{f}^{2}}{\phi } \right)I}^{2} \right|{{\mathsf{\mathcal{F}}}_{k}} \right] \\
 &   +\mathbb{E}\left[ \left. \left\| {{x}_{k+1}}-{{x}^{*}} \right\|_{2\left( {{W}^{2}}-I+{{L}^{2}} \right)-{{L}^{2}}-\alpha \left( 2\mu -\phi  \right)I}^{2} \right|{{\mathsf{\mathcal{F}}}_{k}} \right] \\
 &   +\frac{4\alpha }{\eta }{{L}_{f}}{{p}_{k}}+\frac{\alpha }{\eta }\left( 2{{L}_{f}}-\mu  \right){{L}_{f}}{{\left\| {{x}_{k}}-{{x}^{*}} \right\|}^{2}}
 \end{split}
\end{flalign}
Adding both sides of \eqref{E24} with $c\left( \mathbb{E}\left[ \left. {{p}_{k+1}} \right|{{\mathsf{\mathcal{F}}}_{k}} \right]-{{p}_{k}} \right)$, $c>0$, and using Lemma 3, we get
\begin{flalign}\label{E25}
  & \mathbb{E}\left[ \left. {{V}_{1}}\left( {{x}_{k+1}},{{\lambda }_{k+1}} \right)+c{{p}_{k+1}} \right|{{\mathsf{\mathcal{F}}}_{k}} \right]-{{V}_{1}}\left( {{x}_{k}},{{\lambda }_{k}} \right)-c{{p}_{k}} \notag\\
\le & -\mathbb{E}\left[ \left. \left\| {{x}_{k+1}}-{{x}_{k}} \right\|_{{{W}^{2}}-\alpha \left( \eta +\frac{L_{f}^{2}}{\phi } \right)I}^{2} \right|{{\mathsf{\mathcal{F}}}_{k}} \right] \notag\\
 &    +\mathbb{E}\left[ \left. \left\| {{x}_{k+1}}-{{x}^{*}} \right\|_{2\left( {{W}^{2}}-I+{{L}^{2}} \right)-{{L}^{2}}-\alpha \left( 2\mu -\phi  \right)I}^{2} \right|{{\mathsf{\mathcal{F}}}_{k}} \right] \tag{24}\\
 &   +\frac{4\alpha }{\eta }{{L}_{f}}{{p}_{k}}+\frac{\alpha }{\eta }\left( 2{{L}_{f}}-\mu  \right){{L}_{f}}{{\left\| {{x}_{k}}-{{x}^{*}} \right\|}^{2}} \notag\\
 &     -\frac{c}{{{q}_{\max }}}{{p}_{k}}+\frac{c}{{{q}_{\min }}}\frac{{{L}_{f}}}{2}{{\left\| {{x}_{k}}-{{x}^{*}} \right\|}^{2}}\notag
\end{flalign}
Next, we add $\mathbb{E}\left[ \left. \left\| {{x}_{k+1}}-{{x}^{*}} \right\|_{\gamma Q}^{2} \right|{{\mathsf{\mathcal{F}}}_{k}} \right]-\left\| {{x}_{k}}-{{x}^{*}} \right\|_{\gamma Q}^{2}$, where $Q>0$ and $0<\gamma <1$, on the both sides of \eqref{E25}.
The proof is completed.
\setcounter{equation}{24}
\end{proof}
\begin{lemma}
Let $\tilde{L}{{\in }\mathbb{R}^{m\times m}}$ be a Laplacian matrix of a connected undirected graph. The following statements hold.

(i) There exists an orthogonal matrix $\Xi =\left[ r \;\; R \right]\in {{\mathbb{R}}^{m\times m}}$ with $r=(1/\sqrt{m}){{1}_{m}}$ satisfying ${{\Xi }^{\text{T}}}{\tilde{L}}\Xi =\rm{diag}\left\{ 0, \Lambda  \right\}$, where $\Lambda $ is a diagonal matrix consisting of nonzero eigenvalues of $\tilde{L}$. In addition, ${{R}^{\rm{T}}}R=I$ and $R{{R}^{\rm{T}}}=I-(1/m){{1}_{m}}1_{m}^{\rm{T}}$.

(ii) ${{x}^{\rm{T}}}\tilde{L}x\ge {{\rho }_{2}}(\tilde{L}){{\left\| x-(1_{m}^{\rm{T}}x/m){{1}_{m}} \right\|}^{2}}$ for any $x\in {{\mathbb{R}}^{m}}$, where ${{\rho }_{2}}(\tilde{L})$ is the smallest nonzero eigenvalue of ${\tilde{L}}$.
\end{lemma}
\begin{proof}
Recalling the definition of Laplacian matrix ${\tilde{L}}$ and $\tilde{L}{{1}_{m}}=0$, we obtain that there exists a matrix $\Xi =[ \frac{1}{\sqrt{m}}{{1}_{m}}, R ]$ such that ${{\Xi }^{\text{T}}}L\Xi =\rm{diag}\left\{ 0, \Lambda  \right\}$ and ${{\Xi }^{\text{T}}}\Xi =\Xi {{\Xi }^{\text{T}}}=I$.
\noindent Let
\begin{flalign*}
\Xi =\left[ \begin{matrix}
   {{r}_{11}} & \cdots  & {{r}_{1m}}  \\
   \vdots  & \ddots  & \vdots   \\
   {{r}_{m1}} & \cdots  & {{r}_{mm}}  \\
\end{matrix} \right] \text{and} \;\; R=\left[ \begin{matrix}
   {{r}_{12}} & \cdots  & {{r}_{1m}}  \\
   \vdots  & \ddots  & \vdots   \\
   {{r}_{m2}} & \cdots  & {{r}_{mm}}  \\
\end{matrix} \right]
\end{flalign*}
where ${{r}_{i1}}=1/\sqrt{m}$, $\forall i=1,\ldots ,m$.
Considering $\Xi {{\Xi }^{\text{T}}}=I$, we have
\begin{flalign*}
{{\left[ \Xi {{\Xi }^{\text{T}}} \right]}_{ij}}=\sum\limits_{l=1}^{m}{{{r}_{il}}{{r}_{jl}}}=\frac{1}{m}+\sum\limits_{l=2}^{m}{{{r}_{il}}{{r}_{jl}}}=\left\{ \begin{matrix}
   1,  \text{if}\; i=j  \\
   0,  \text{if}\; i\ne j  \\
\end{matrix} \right.
\end{flalign*}
which means that, for all $i,j=1,\ldots ,m$, $\sum\nolimits_{l=2}^{m}{{{r}_{il}}{{r}_{jl}}}=1-1/m$ if $i=j$; and $\sum\nolimits_{l=2}^{m}{{{r}_{il}}{{r}_{jl}}}=-1/m$, otherwise, i.e., $R{{R}^{\text{T}}}={{I}}-(1/m){{1}_{m}}1_{m}^{\text{T}}$.
Similarly, for all $i,j=2,\ldots ,m$, we have $\sum\nolimits_{l=1}^{m}{{{r}_{li}}{{r}_{lj}}}=1$ if $i=j$; and $\sum\nolimits_{l=1}^{m}{{{r}_{li}}{{r}_{lj}}}=0$ if $i\ne j$, i.e., ${{R}^{\text{T}}}R={{I}}$.
The proof of Lemma 5 ($i$) is completed.

Considering the definitions of $\Xi $ and $\Lambda $, one has
$\Xi {\rm{diag}}\left\{ 0, {{\Lambda }^{\frac{1}{2}}} \right\}{\rm{diag}}\left\{ 0, {{\Lambda }^{\frac{1}{2}}} \right\}{{\Xi }^{\text{T}}}=R{{\Lambda }^{\frac{1}{2}}}{{\Lambda }^{\frac{1}{2}}}R$.
Then, we have
\begin{flalign*}
{{x}^{\text{T}}}\tilde{L}x=&{{x}^{\text{T}}}\left( R\Lambda {{R}^{\text{T}}} \right)x\\
 \ge &{{\rho }_{2}}( {\tilde{L}} ){{x}^{\text{T}}}\left( R{{R}^{\text{T}}} \right)x\\
 =&{{\rho }_{2}}( {\tilde{L}} ){{\left\| x-(1_{m}^{\text{T}}x/m){{1}_{m}} \right\|}^{2}}
\end{flalign*}
The proof of Lemma 5 ($ii$) is completed
\end{proof}
\begin{theorem}
Consider the algorithm in (16) and let  the required conditions in Lemmas 1-5 be satisfied.
If the parameters $\eta $ and $c$ satisfy
\begin{flalign}
& \eta \in \left( \frac{2\frac{{{L}_{f}}}{{{q}_{\min }}}{{q}_{\max }}{{L}_{f}}+\left( 2{{L}_{f}}-\mu  \right){{L}_{f}}}{\gamma \left( 2\mu -\phi  \right)},\infty  \right)\label{E26}\\
& c\in \left( \frac{\gamma \alpha \left( 2\mu -\phi  \right)-\frac{\alpha \left( 2{{L}_{f}}-\mu  \right){{L}_{f}}}{\eta }}{\frac{{{L}_{f}}}{2{{q}_{\min }}}},\frac{4\alpha }{\eta }{{q}_{\max }}{{L}_{f}} \right)\label{E27}
\end{flalign}
and the step-size $\alpha$ is selected from the interval
\begin{flalign}\label{E28}
\alpha \in \left( 0,\frac{{{\left[ {{\rho }_{\min }}\left( W \right) \right]}^{2}}}{\eta +\frac{L_{f}^{2}}{\phi }} \right)
\end{flalign}
where $0<\phi<2\mu$, the global variable, $x_{k}$, generated by Algorithm 1 almost surely converges to $x^{*}$ with a linear convergence rate $O( {{( 1+\delta  )}^{-k}} )$, i.e., ${{\left\| {{x}_{k+1}}-{{x}^{*}} \right\|}^{2}}\le {{(1+\delta )}^{-1}}{{\left\| {{x}_{k}}-{{x}^{*}} \right\|}^{2}}$, where
\begin{flalign*}
\begin{split}
0<&\delta <\Theta := \min \left\{ \frac{{{\left[ {{\rho }_{\min }}\left( W \right) \right]}^{2}}-\alpha \left( \eta +\frac{L_{f}^{2}}{\phi } \right)}{\frac{1}{{{\rho }_{2}}\left( {{L}^{2}} \right)}\frac{d}{d-1}e} \right., \\
 &\frac{\left( 1-\gamma  \right)\alpha \left( 2\mu -\phi  \right)}{1+\gamma {{\lambda }_{\max }}\left( Q \right)+\frac{4}{{{\rho }_{2}}\left( {{L}^{2}} \right)}d{{\left( {{\max }_{i}}\left\{ {{\rho }_{i}}\left( W \right)\left( {{\rho }_{i}}\left( W \right)-1 \right) \right\} \right)}^{2}}}, \\
 &\qquad\qquad \left. \frac{\gamma \alpha \left( 2\mu -\phi  \right)-\frac{\alpha \left( 2{{L}_{f}}-\mu  \right){{L}_{f}}}{\eta }-\frac{c{{L}_{f}}}{2{{q}_{\min }}}}{\frac{c}{{{q}_{\min }}}\frac{{{L}_{f}}}{2}+\frac{1}{{{\rho }_{2}}\left( {{L}^{2}} \right)}\frac{d}{d-1}\frac{e}{e-1}{{\alpha }^{2}}\left( 2{{L}_{f}}-\mu  \right){{L}_{f}}} \right\}
 \end{split}
\end{flalign*}
when $e>1$ and $d>1$.
\end{theorem}
\begin{proof}
Defining ${{V}_{k}}={{V}_{1}}\left( {{x}_{k}},{{\lambda }_{k}} \right)+\left\| {{x}_{k}}-{{x}^{*}} \right\|_{\gamma Q}^{2}+c{{p}_{k}}$, the global variables, $x_k$, generated by Algorithm 1 almost surely converges to the global optimal solution $x^{*}$ with a linear convergence rate $O( {{( 1+\delta  )}^{-k}} )$ if there exist a positive $\delta$ such that $\mathbb{E}\left[ \left. {{V}_{k+1}} \right|{{\mathsf{\mathcal{F}}}_{k}} \right]-{{V}_{k}} \le -{{\Delta }_{k+1}} \le -\delta \mathbb{E}\left[ \left. {{V}_{k+1}} \right|{{\mathsf{\mathcal{F}}}_{k}} \right]$, i.e., $\delta \mathbb{E}\left[ \left. {{V}_{k+1}} \right|{{\mathsf{\mathcal{F}}}_{k}} \right]\le {{\Delta }_{k+1}}$.
Next, we study a quantitative description of the convergence rate $\delta$ which ensures the linear convergence rate of the S-DIGing algorithm.

To obtain $\delta \mathbb{E}\left[ \left. {{V}_{k+1}} \right|{{\mathsf{\mathcal{F}}}_{k}} \right]\le {{\Delta }_{k+1}}$, it is sufficient to prove
\begin{flalign}\label{E29}
&\delta \mathbb{E}\left[ \left. \left\| {{x}_{k+1}}-{{x}^{*}} \right\|_{{{W}^{2}}+\gamma Q}^{2}+{{\left\| {{\lambda }_{k+1}}-{{\lambda }^{*}} \right\|}^{2}}+c{{p}_{k+1}} \right|{{\mathsf{\mathcal{F}}}_{k}} \right] \notag\\
\le &\mathbb{E}\left[ \left. \left\| {{x}_{k+1}}-{{x}_{k}} \right\|_{{{W}^{2}}-\alpha \left( \eta +\frac{L_{f}^{2}}{\phi } \right)I}^{2} \right|{{\mathsf{\mathcal{F}}}_{k}} \right] \tag{28}\\
 & -\mathbb{E}\left[ \left. \left\| {{x}_{k+1}}-{{x}^{*}} \right\|_{2\left( {{W}^{2}}-I+{{L}^{2}} \right)-{{L}^{2}}-\alpha \left( 2\mu -\phi  \right)I+\gamma Q}^{2} \right|{{\mathsf{\mathcal{F}}}_{k}} \right]  \notag\\
 & -\left\| {{x}_{k}}-{{x}^{*}} \right\|_{\frac{\alpha }{\eta }\left( 2{{L}_{f}}-\mu  \right){{L}_{f}}I+\frac{c{{L}_{f}}}{2{{q}_{\min }}}I-\gamma Q}^{2}-( \frac{4\alpha }{\eta }{{L}_{f}}-\frac{c}{{{q}_{\max }}} ){{p}_{k}}\notag
\end{flalign}
\setcounter{equation}{28}
\noindent Using Lemma 3 and Assumption 2, we get
\begin{flalign}
\begin{split}
  & \delta c\mathbb{E}\left[ \left. {{p}_{k+1}} \right|{{\mathsf{\mathcal{F}}}_{k}} \right] \\
 \le & \delta c( 1-\frac{1}{{{q}_{\max }}} ){{p}_{k}}\\
 &+\frac{\delta c}{{{q}_{\min }}}\left( f({{x}_{k}})-f({{x}^{*}})-{{\left( \nabla f({{x}^{*}}) \right)}^{\text{T}}}\left( {{x}_{k}}-{{x}^{*}} \right) \right) \\
 \le & \delta c( 1-\frac{1}{{{q}_{\max }}} ){{p}_{k}}+\frac{\delta c}{{{q}_{\min }}}\frac{{{L}_{f}}}{2}{{\left\| {{x}_{k}}-{{x}^{*}} \right\|}^{2}}
 \end{split}
\end{flalign}
Then, a sufficient condition for Eq. \eqref{E29} to be held is
\begin{flalign}\label{E31}
  &   \delta \mathbb{E}\left[ \left. {{\left\| {{\lambda }_{k+1}}-{{\lambda }^{*}} \right\|}^{2}} \right|{{\mathsf{\mathcal{F}}}_{k}} \right] \notag\\
\le & \mathbb{E}\left[ \left. \left\| {{x}_{k+1}}-{{x}_{k}} \right\|_{{{W}^{2}}-\alpha \left( \eta +\frac{L_{f}^{2}}{\phi } \right)I}^{2} \right|{{\mathsf{\mathcal{F}}}_{k}} \right] \notag\\
&+\mathbb{E}\left[ \left. \left\| {{x}_{k+1}}-{{x}^{*}} \right\|_{-2\left( {{W}^{2}}-I+{{L}^{2}} \right)+{{L}^{2}}+\alpha \left( 2\mu -\phi  \right)I-\gamma Q}^{2} \right|{{\mathsf{\mathcal{F}}}_{k}} \right]  \notag\\
 &+\mathbb{E}\left[ \left. \left\| {{x}_{k+1}}-{{x}^{*}} \right\|_{-\delta \left( {{W}^{2}}+\gamma Q \right)}^{2} \right|{{\mathsf{\mathcal{F}}}_{k}} \right] \tag{30}\\
  &+\left\| {{x}_{k}}-{{x}^{*}} \right\|_{\gamma Q-\left( \frac{\alpha }{\eta }\left( 2{{L}_{f}}-\mu  \right){{L}_{f}}+\frac{c{{L}_{f}}}{2{{q}_{\min }}} \right)I-\frac{\delta c{{L}_{f}}}{2{{q}_{\min }}}I}^{2} \notag\\
 &+\left( \frac{c}{{{q}_{\max }}}-\frac{4\alpha }{\eta }{{L}_{f}}-\delta c( 1-\frac{1}{{{q}_{\max }}} ) \right){{p}_{k}} \notag
\end{flalign}
Observing the above inequality, we find that there is only $\delta {{\left\| {{\lambda }_{k+1}}-{{\lambda }^{*}} \right\|}^{2}}$ on the left side.
It is difficult to directly analyze the conditions which make the inequality held.
Thus we establish an upper bound of the left hand of \eqref{E31} which is lower than the right hand of \eqref{E31}.
To this end, we use Eq. (19) and the basic inequality: ${{\left\| a+b \right\|}^{2}}\le \tau  {{\left\| a \right\|}^{2}}+{\tau }/(\tau -1){{\left\| b \right\|}^{2}}$, $\forall a,b\in {{\mathbb{R}}^{n}}$, $\tau >1$, to obtain
\setcounter{equation}{30}
\begin{flalign}
\begin{split}
  &{{\left\| L\left( {{\lambda }_{k+1}}-{{\lambda }^{*}} \right) \right\|}^{2}} \\
=& \left\| \left( {{W}^{2}}-I+{{L}^{2}} \right)\left( {{x}_{k+1}}-{{x}^{*}} \right)-{{W}^{2}}\left( {{x}_{k+1}}-{{x}_{k}} \right) \right. \\
 & {{\left. -\alpha \left( {{g}_{k}}-\nabla F\left( {{x}^{*}} \right) \right) \right\|}^{2}} \\
 \le & d{{\left\| \left( {{W}^{2}}-I+{{L}^{2}} \right)\left( {{x}_{k+1}}-{{x}^{*}} \right) \right\|}^{2}} \\
 & +\frac{d}{d-1}e{{\left\| {{W}^{2}}\left( {{x}_{k+1}}-{{x}_{k}} \right) \right\|}^{2}} \\
 & +\frac{d}{d-1}\frac{e}{e-1}{{\alpha }^{2}}{{\left\| {{g}_{k}}-\nabla F\left( {{x}^{*}} \right) \right\|}^{2}}
\end{split}
\end{flalign}
where $e>1$ and $d>1$.

Computing the conditional expectation on ${{\mathsf{\mathcal{F}}}_{k}}$ and using Lemma 2, we have
\begin{flalign}\label{E33}
\begin{split}
  & \mathbb{E}\left[ \left. {{\left\| L\left( {{\lambda }_{k+1}}-{{\lambda }^{*}} \right) \right\|}^{2}} \right|{{\mathsf{\mathcal{F}}}_{k}} \right] \\
\le & \mathbb{E}\left[ \left. \left\| {{x}_{k+1}}-{{x}^{*}} \right\|_{d{{\left( {{W}^{2}}-I+{{L}^{2}} \right)}^{2}}}^{2} \right|{{\mathsf{\mathcal{F}}}_{k}} \right] \\
 &    +\frac{d}{d-1}e\mathbb{E}\left[ \left. \left\| {{x}_{k+1}}-{{x}_{k}} \right\|_{{{W}^{4}}}^{2} \right|{{\mathsf{\mathcal{F}}}_{k}} \right] \\
& +\frac{d}{d-1}\frac{e}{e-1}{{\alpha }^{2}}\left( 2{{L}_{f}}-\mu  \right){{L}_{f}}{{\left\| {{x}_{k}}-{{x}^{*}} \right\|}^{2}} \\
 & +4\frac{d}{d-1}\frac{e}{e-1}{{\alpha }^{2}}{{L}_{f}}{{p}_{k}}
 \end{split}
 \end{flalign}
Note that the expectations $\mathbb{E}\left[ \left. {{\left\| {{x}_{k}}-{{x}^{*}} \right\|}^{2}} \right|{{\mathsf{\mathcal{F}}}_{k}} \right]={{\left\| {{x}_{k}}-{{x}^{*}} \right\|}^{2}}$ and $\mathbb{E}\left[ \left. {{p}_{k}} \right|{{\mathsf{\mathcal{F}}}_{k}} \right]={{p}_{k}}$ due to $x_{k}$ and $p_{k}$ are determined estimate at time instant $k$.

Leveraging Lemma 5 and substituting the term ${{\left\| L\left( {{\lambda }_{k+1}}-{{\lambda }^{*}} \right) \right\|}^{2}}$ of Eq. \eqref{E33} by its lower bound ${{\rho }_{2}}\left( {{L}^{2}} \right){{\left\| {{\lambda }_{k+1}}-{{\lambda }^{*}} \right\|}^{2}}$, we obtain
\begin{flalign}\label{E34}
  &{{\rho }_{2}}\left( {{L}^{2}} \right)\mathbb{E}\left[ \left. {{\left\| {{\lambda }_{k+1}}-{{\lambda }^{*}} \right\|}^{2}} \right|{{\mathsf{\mathcal{F}}}_{k}} \right] \notag\\
 \le & \mathbb{E}\left[ \left. \left\| {{x}_{k+1}}-{{x}^{*}} \right\|_{d{{\left( {{W}^{2}}-I+{{L}^{2}} \right)}^{2}}}^{2} \right|{{\mathsf{\mathcal{F}}}_{k}} \right] \notag\\
 &    +\frac{d}{d-1}e\mathbb{E}\left[ \left. \left\| {{x}_{k+1}}-{{x}_{k}} \right\|_{{{W}^{4}}}^{2} \right|{{\mathsf{\mathcal{F}}}_{k}} \right] \tag{33}\\
& +\frac{d}{d-1}\frac{e}{e-1}{{\alpha }^{2}}\left( 2{{L}_{f}}-\mu  \right){{L}_{f}}{{\left\| {{x}_{k}}-{{x}^{*}} \right\|}^{2}} \notag\\
 & +4\frac{d}{d-1}\frac{e}{e-1}{{\alpha }^{2}}{{L}_{f}}{{p}_{k}} \notag
\end{flalign}
\setcounter{equation}{33}
\noindent Combing Eqs. \eqref{E31} and \eqref{E34}, the sufficient condition for Eq. \eqref{E29} can be rewritten as
\begin{flalign}\label{E35}
&\frac{\delta }{{{\rho }_{2}}\left( {{L}^{2}} \right)}\frac{d}{d-1}e\mathbb{E}\left[ \left. \left\| {{x}_{k+1}}-{{x}_{k}} \right\|_{{{W}^{4}}}^{2} \right|{{\mathsf{\mathcal{F}}}_{k}} \right] \notag\\
  &+ \frac{\delta }{{{\rho }_{2}}\left( {{L}^{2}} \right)}\mathbb{E}\left[ \left. \left\| {{x}_{k+1}}-{{x}^{*}} \right\|_{d{{\left( {{W}^{2}}-I+{{L}^{2}} \right)}^{2}}}^{2} \right|{{\mathsf{\mathcal{F}}}_{k}} \right] \notag\\
 & +\frac{4\delta }{{{\rho }_{2}}\left( {{L}^{2}} \right)}\frac{d}{d-1}\frac{e}{e-1}{{\alpha }^{2}}{{L}_{f}}{{p}_{k}} \notag\\
 & +\frac{\delta }{{{\rho }_{2}}\left( {{L}^{2}} \right)}\frac{d}{d-1}\frac{e}{e-1}{{\alpha }^{2}}\left( 2{{L}_{f}}-\mu  \right){{L}_{f}}{{\left\| {{x}_{k}}-{{x}^{*}} \right\|}^{2}} \notag\\
 \le & \mathbb{E}\left[ \left. \left\| {{x}_{k+1}}-{{x}_{k}} \right\|_{{{W}^{2}}-\alpha \left( \eta +\frac{L_{f}^{2}}{\phi } \right)I}^{2} \right|{{\mathsf{\mathcal{F}}}_{k}} \right] \notag\\
&+\mathbb{E}\left[ \left. \left\| {{x}_{k+1}}-{{x}^{*}} \right\|_{-2\left( {{W}^{2}}-I+{{L}^{2}} \right)+{{L}^{2}}+\alpha \left( 2\mu -\phi  \right)I-\gamma Q}^{2} \right|{{\mathsf{\mathcal{F}}}_{k}} \right] \notag\\
 &+\mathbb{E}\left[ \left. \left\| {{x}_{k+1}}-{{x}^{*}} \right\|_{-\delta \left( {{W}^{2}}+\gamma Q \right)}^{2} \right|{{\mathsf{\mathcal{F}}}_{k}} \right] \tag{34}\\
 & +\left\| {{x}_{k}}-{{x}^{*}} \right\|_{\gamma Q-\left( \frac{\alpha }{\eta }\left( 2{{L}_{f}}-\mu  \right){{L}_{f}}+\frac{c{{L}_{f}}}{2{{q}_{\min }}} \right)I-\frac{\delta c{{L}_{f}}}{2{{q}_{\min }}}I}^{2} \notag\\
 & +\left( \frac{c}{{{q}_{\max }}}-\frac{4\alpha }{\eta }{{L}_{f}}-\delta c( 1-\frac{1}{{{q}_{\max }}} ) \right){{p}_{k}}\notag
\end{flalign}
\setcounter{equation}{34}
which is equivalent to prove
\begin{flalign}
  \begin{split}
  0\le & {{W}^{2}}-\alpha ( \eta +\frac{L_{f}^{2}}{\phi } )I-\frac{\delta }{{{\rho }_{2}}\left( {{L}^{2}} \right)}\frac{d}{d-1}e{{W}^{4}}\label{E36}
  \end{split}\\
 \begin{split}\label{E37}
 0\le & -\left[ 2\left( {{W}^{2}}-I+{{L}^{2}} \right)-{{L}^{2}}-\alpha \left( 2\mu -\phi  \right)I+\gamma Q \right]\\
 &-\delta {{W}^{2}}-\delta \gamma Q-\frac{\delta }{{{\rho }_{2}}\left( {{L}^{2}} \right)}d{{\left( {{W}^{2}}-I+{{L}^{2}} \right)}^{2}}
 \end{split}\\
 \begin{split}\label{E38}
 0\le & \gamma Q-\left( \frac{\alpha}{\eta }\left( 2{{L}_{f}}-\mu  \right){{L}_{f}}+\frac{c{{L}_{f}}}{2{{q}_{\min }}} \right)I-\frac{\delta c L_{f}}{{2{q}_{\min }}}I\\
 &-\frac{\delta }{{{\rho }_{2}}\left( {{L}^{2}} \right)}\frac{d}{d-1}\frac{e}{e-1}{{\alpha }^{2}}\left( 2{{L}_{f}}-\mu  \right){{L}_{f}}I
 \end{split}\\
 \begin{split}\label{E39}
 0\le & \frac{c}{{{q}_{\max }}}-\frac{4\alpha }{\eta }{{L}_{f}}-\frac{4\delta }{{{\rho }_{2}}\left( {{L}^{2}} \right)}\frac{d}{d-1}\frac{e}{e-1}{{\alpha }^{2}}{{L}_{f}}\\
 &-\delta c( 1-\frac{1}{{{q}_{\max }}} )
 \end{split}
 \end{flalign}
It can be verified that \eqref{E36} is tenable if $\alpha$ and $\delta$ satisfy
 \begin{flalign}
 0<&\alpha<\frac{{{\left[ {{\rho }_{\min }}\left( W \right) \right]}^{2}}}{\eta +\frac{L_{f}^{2}}{\phi }}\\
 0<&\delta < \frac{{{\left[ {{\rho }_{\min }}\left( W \right) \right]}^{2}}-\alpha \left( \eta +\frac{L_{f}^{2}}{\phi } \right)}{\frac{1}{{{\rho }_{2}}\left( {{L}^{2}} \right)}\frac{d}{d-1}e}
 \end{flalign}
Recalling that $Q=\left( I+3W \right)\left( I-W \right)+\alpha \left( 2\mu -\phi  \right)I$. Note that ${{\rho }_{\min }}\left( Q \right)=\alpha \left( 2\mu -\phi  \right)$, ${{\rho }_{\max }}\left( Q \right)={{\max }_{i}}\left\{ \left( 1+3{{\rho }_{i}}\left( W \right) \right)\left( 1-{{\rho }_{i}}\left( W \right) \right) \right\}+\alpha \left( 2\mu -\phi  \right)$ and ${{\rho }_{\max }}\left( W\left( W-I \right) \right)={{\max }_{i}}\left\{ {{\rho }_{i}}\left( W \right)\left( {{\rho }_{i}}\left( W \right)-1 \right) \right\}$.
Then, condition \eqref{E37} can be satisfied if there exists a positive constant $\delta$ such that
\begin{flalign}\label{44}
\delta < \frac{\left( 1-\gamma  \right)\alpha \left( 2\mu -\phi  \right)}{1+\gamma {{\rho }_{\max }}\left( Q \right)+\frac{4}{{{\rho }_{2}}\left( {{L}^{2}} \right)}d{{\left( {{\max }_{i}}\left\{ {{\rho }_{i}}\left( W \right)\left( {{\rho }_{i}}\left( W \right)-1 \right) \right\} \right)}^{2}}}
\end{flalign}
Rearranging the terms in \eqref{E38}, we obtain
\begin{flalign}\label{44}
\begin{split}
  & \gamma \alpha \left( 2\mu -\phi  \right)-\frac{\alpha \left( 2{{L}_{f}}-\mu  \right){{L}_{f}}}{\eta }-\frac{c{{L}_{f}}}{2{{q}_{\min }}} \\
\ge &\frac{\delta c}{{{q}_{\min }}}\frac{{{L}_{f}}}{2}+\frac{\delta }{{{\rho }_{2}}\left( {{L}^{2}} \right)}\frac{d}{d-1}\frac{e}{e-1}{{\alpha }^{2}}\left( 2{{L}_{f}}-\mu  \right){{L}_{f}}
 \end{split}
\end{flalign}
where $0<\phi<2\mu$.
It is clear that we can choose a small enough nonnegative constant $\delta$ such that \eqref{44} holds if its left hand side is positive.
To this end, we choose the parameters $\eta$ and $c$ satisfying
\begin{flalign}
\eta > & \frac{\left( 2{{L}_{f}}-\mu  \right){{L}_{f}}}{\gamma \left( 2\mu -\phi  \right)}\\
0 <&c < 2{{q}_{\min }}\frac{\gamma \alpha \left( 2\mu -\phi  \right)-\frac{\alpha \left( 2{{L}_{f}}-\mu  \right){{L}_{f}}}{\eta }}{{{L}_{f}}}
\end{flalign}
Then, we have
\begin{flalign}
\delta < \frac{\gamma \alpha \left( 2\mu -\phi  \right)-\frac{\alpha \left( 2{{L}_{f}}-\mu  \right){{L}_{f}}}{\eta }-\frac{c{{L}_{f}}}{2{{q}_{\min }}}}{\frac{c}{{{q}_{\min }}}\frac{{{L}_{f}}}{2}+\frac{1}{{{\rho }_{2}}\left( {{L}^{2}} \right)}\frac{d}{d-1}\frac{e}{e-1}{{\alpha }^{2}}\left( 2{{L}_{f}}-\mu  \right){{L}_{f}}}
\end{flalign}
Similarly, we have a sufficient condition for \eqref{E39}
\begin{flalign}
c>&\frac{4\alpha }{\eta }{{q}_{\max }}{{L}_{f}}\\
\eta > & \frac{ 2\frac{{{L}_{f}}}{{{q}_{\min }}}{{q}_{\max }}{{L}_{f}}+\left( 2{{L}_{f}}-\mu  \right){{L}_{f}} }{\gamma \left( 2\mu -\phi  \right)}\\
0<&\delta <\frac{\gamma \alpha (2\mu -\phi )-\frac{\alpha }{\eta }(2{{L}_{f}}-\mu ){{L}_{f}}-\frac{c{{L}_{f}}}{2{{q}_{\min }}}}{\frac{c{{L}_{f}}}{2{{q}_{\min }}}+\frac{1}{{{\rho }_{2}}({{L}^{2}})}\frac{d}{d-1}\frac{e}{e-1}{{\alpha }^{2}}(2{{L}_{f}}-\mu ){{L}_{f}}}
\end{flalign}
Concluding above analysis, we get that the $x_{k}$ generated by Algorithm 1 converges to $x^{*}$ with linear rate $O( {{\left( 1+\delta  \right)}^{-k}} )$ if conditions \eqref{E26}, \eqref{E27}, \eqref{E28}, $0<\phi<2\mu$ and $0< \gamma < 1$ are satisfied.
The proof is completed.
\end{proof}
\begin{remark}
From the above analysis, it is proved that $\left\| {{x}_{k}}-{{x}^{*}} \right\|^{2}\le {{\left( 1+\delta  \right)}^{-k}}\left\| {{x}_{0}}-{{x}^{*}} \right\|^{2}=\kappa {{\left( 1-q \right)}^{k}}$ where $q=\frac{\delta}{1+\delta}$ and $\kappa = \left\| {{x}_{0}}-{{x}^{*}} \right\|^{2}$.
Due to $1-q\le {{e}^{-q}}$ where $0<q<1$, we have $\kappa {{\left( 1-q \right)}^{k}}\le \kappa {{e}^{-qk}}$.
In order to get ${{\left\| {{x}_{k}}-{{x}^{*}} \right\|}^{2}}\le \varepsilon $, S-DIGing algorithm needs iterations number $k\ge \frac{1}{q}\log \frac{\kappa }{\varepsilon }=\left( 1+\frac{1}{\delta } \right)\log \frac{\kappa }{\varepsilon }>\left( 1+\frac{1}{\Theta } \right)\left( \log \kappa +\log \frac{1}{\varepsilon } \right)$ where $\Theta$ is defined in Theorem 1.
From the analysis of DIGing algorithm in \cite{Nedic2017c}, it can be concluded that the iteration number $k\ge \left( 1\text{+}\frac{1}{\Xi } \right)\left( \log\kappa +\log \frac{1}{\varepsilon } \right)$ where $\Xi =\frac{\alpha \mu }{1.5-\alpha \mu }$ is needed for $\left\| {{x}_{k}}-{{x}^{*}} \right\|^{2}\le \varepsilon$.
Because there are too many parameters in the convergence rate, it is difficult to compare the complexity of the two algorithms directly.
Therefore, the simulation part will show the number of iterations and time required for the two algorithms to achieve the same residual error to compare the complexity of the two algorithms.
\end{remark}

\section{Numerical examples}
In this section, we provide some numerical examples about logistic regression, energy-based source localization, and $K$-means clustering to show the effectiveness of the S-DIGing algorithm.
We use CVX \cite{Grant2017} to work out the global optimal solution ${{\tilde{x}}^{*}}$ by solving the problem in a centralized way,
In the following simulations, the residual is defined as ${{\log }_{10}}\left( \left( 1/m \right)\sum\nolimits_{i=1}^{m}{\left\| x_{k}^{i}-{{{\tilde{x}}}^{*}} \right\|} \right)$.
\subsection{Distributed Logistic Regression}
In this subsection, we leverage the S-DIGing algorithm to solve a binary classification problem by logistic regression and study the performance of the algorithm under different settings.
We assign $N=\sum\nolimits_{i=1}^{m}{{{q}_{i}}}$ samples to $m$ agents, and each one gets $q_{i}$ samples.
We assume that the samples are distributed equally over the agents, i.e., $q_{i}=N/m$, $\forall i \in \mathsf{\mathcal{V}}$.
Then we employ $m$ agents of an undirected network to cooperatively solve the following distributed logistic regression problem:
\begin{flalign}\label{E50}
{{\tilde{x}}^{*}}=\arg \underset{\tilde{x}\in {{\mathbb{R}}^{n}}}{\mathop{\min }}\left(\frac{\lambda }{2}{{\left\| {\tilde{x}} \right\|}^{2}}+\sum\limits_{i=1}^{m}{\sum\limits_{h=1}^{{{q}_{i}}}{\log \left( 1+\exp \left( -{{l}_{i,h}}c_{i,h}^{\text{T}}\tilde{x} \right) \right)}}\right)
\end{flalign}
where $l_{i,h}\in \{-1,+1\}$ and $c_{i,h}$ are label and training data of $h$-th sample kept by agent $i$, respectively.
The regularization term $\left( \lambda /2m \right){{\left\| {\tilde{x}} \right\|}^{2}}$ is added to avoid over-fitting.
Based on previous analysis, the problem in \eqref{E50} can be written in the form of (1) by defining the local objective functions $f_{i}$ as:
\begin{flalign}\label{E51}
{{f}_{i}}\left( {\tilde{x}} \right)=\frac{\lambda }{2m}{{\left\| {\tilde{x}} \right\|}^{2}}+\sum\limits_{i=1}^{{{q}_{i}}}{\log \left( 1+\exp \left( -{{l}_{i,h}}c_{i,h}^{\text{T}}\tilde{x} \right) \right)}
\end{flalign}
where
\begin{flalign}\label{E52}
f_{i}^{h}\left( {\tilde{x}} \right)=\frac{\lambda }{2m}{{\left\| {\tilde{x}} \right\|}^{2}}+{{q}_{i}}\log \left( 1+\exp \left( -{{l}_{i,h}}c_{i,h}^{\text{T}}\tilde{x} \right) \right)
\end{flalign}
for $h = 1,2,\dots,q_{i}$.
Consider the definitions of \eqref{E51} and \eqref{E52}, problem \eqref{E50} can be addressed by S-DIGing algorithm.
\subsubsection{Comparison}
In this case, we solve the logistic regression problem in \eqref{E50} for the mushroom data set provided in UCI Machine Learning Repository\cite{Dua2019}.
A subset of $8000$ samples from the data set are randomly chosen, where $N=6000$ samples are used to train the discriminator ${\tilde{x}}$ and 2000 samples are used for testing. 
Each samples have 22 attributes included cap-shape, cap-surface, cap-color, bruises and so on, 
but the original $12$-th attribute (stalk-surface-above-ring) has missing values and is not used.
Employing the one-hot coding method, the dimension $n$ of each sample is extended to 112.
We choose $m=20$, $q_{i}=300$ for $i=1,2,\dots,m$, and the step-size $\alpha=0.001$.
We let the label $l_{i,h}=+1$ if the sample $c_{i,h}$ is poisonous and the label $l_{i,h}=-1$ if the sample $c_{i,h}$ is eatable.
We compare the performance of the S-DIGing algorithm and the DIGing algorithm.
Fig. 1 shows the evolutions of residuals respect to different algorithms while Fig. 2 shows the testing accuracy.
\begin{figure}
\centering
\subfigure[Evolution of residuals with number of iterations]{
\includegraphics[width=8cm,height=5.5cm]{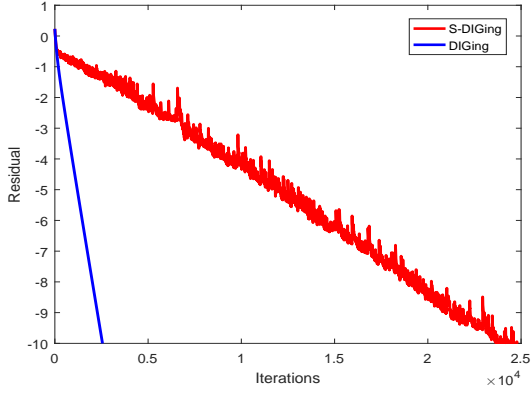}}
\subfigure[Evolution of residuals with running time of algorithms]{
\includegraphics[width=8cm,height=5.5cm]{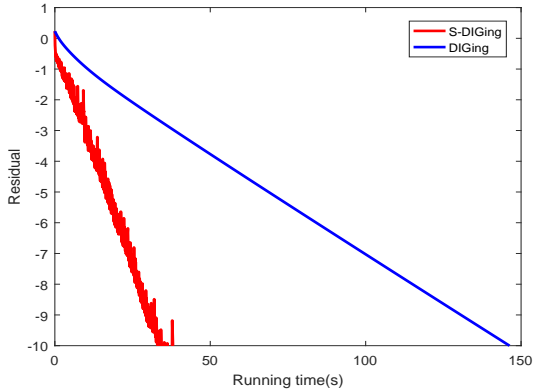}}
\caption{Comparison across S-DIGing and DIGing.}
\end{figure}
From Fig. 1, we find that the S-DIGing algorithm needs more iterations than the DIGing algorithm to achieve a same residual.
However, it is worth to note that the S-DIGing algorithm has an advantage in running time on account of the smaller computational cost required for a single iteration.
\subsubsection{Effects of Fixed Step-Size and Scales of Network}
Firstly, we compare the performance of the S-DIGing algorithm in terms of step-size selection.
We choose $m=100$, $q_{i}=60$ for $i=1,2,\dots,m$, and $n=4$.
For each agent $i$, half on the feature vectors $c_{i,h}\in{{\mathbb{R}}^{4}}$ with label $l_{i,h} = +1$ are drawn by i.i.d $\mathsf{\mathcal{N}}( {{\left[ 2,2,-2,-2 \right]}^{\text{T}}},2I )$ while the others with label $l_{i,h}=-1$ are set to be i.i.d $\mathsf{\mathcal{N}}( {{\left[-2,-2,2,2 \right]}^{\text{T}}},2I )$.
Letting the step-size, $\alpha$, equal to 0.002, 0.006, 0.010, 0.014, 0.018 and 0.022, respectively, Fig. 3 shows the evolutions of residuals respect to diverse step-sizes.
The simulation is performed on the network shown in Fig. 4.
From Fig. 3, we can find that the increasing of step-size plays a positive role in the execution of the S-DIGing algorithm within a certain range.
If the step-size out of the range, the convergence of the S-DIGing algorithm will be deteriorated.
Secondly, we choose $m=50,75,100$ and select $\alpha=0.001$ and $q_{i}=6000/m$ to observe the performance of the algorithm under different scales of networks.
Fig. 5 displays the  evolutions of residuals respect to different scales of network.
\begin{figure}[!htbp]
	\centering\resizebox{8cm}{5.5cm} {\includegraphics*{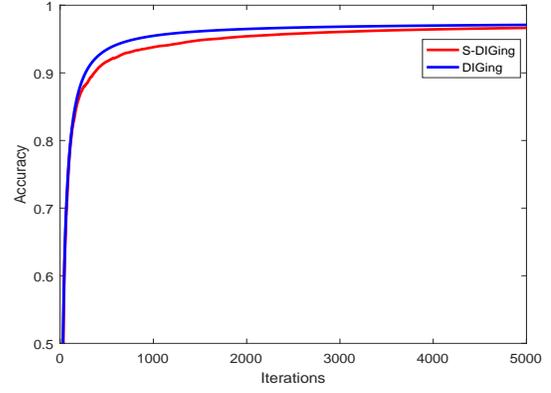}}
	\caption{Testing accuracy.}
\end{figure}
\begin{figure}[!htbp]
	\centering\resizebox{8cm}{5.5cm} {\includegraphics*{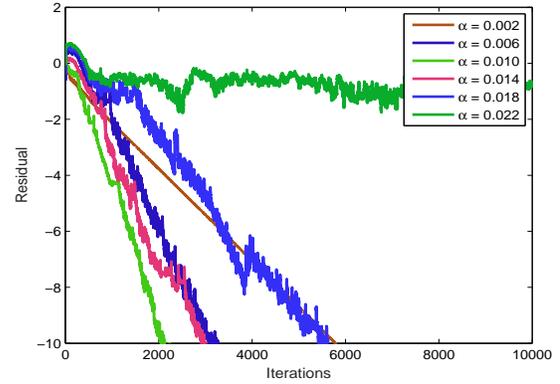}}
	\caption{Comparison across different step-sizes.}
\end{figure}
\begin{figure}[!htbp]
	\centering\resizebox{8cm}{5.5cm} {\includegraphics*{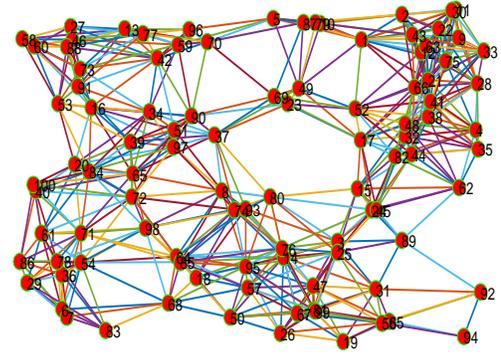}}
	\caption{Network topology.}
\end{figure}
\begin{figure}[!htbp]
	\centering\resizebox{8cm}{5.5cm} {\includegraphics*{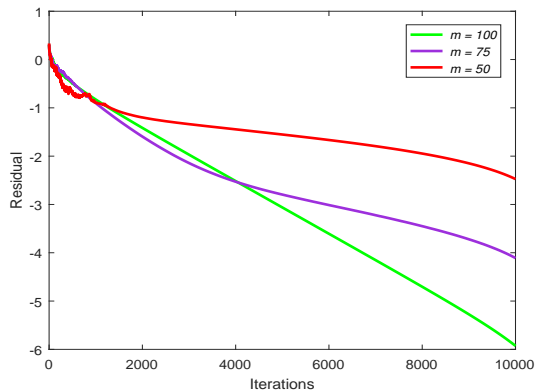}}
	\caption{Comparison across different scales of network.}
\end{figure}
\subsubsection{Image Recognition}
In this case, we solve the logistic regression problem in \eqref{E50} for the MNIST database of handwritten digits provided in \cite{LeCun2010}.
We randomly choose a subset of $58000$ handwritten digits from the MNIST database, where $N=50000$ samples from training set are used to train the discriminator $x$ and 8000 samples from testing set are used for testing.
A part of training samples are shown in Fig. 6.
The network used to solve this problem consists of $m=10$ agents and the probability of connection between each pair of agents is $40\%$.
Each image, $c_{i,h}\in {{\mathbb{R}}^{784}}$, is a vector and the total images are divided among $m$ agents such that each agent has $q_{i}$ = 5000, $i=1,2,\dots,m$, images.
Due to privacy and communication restrictions, agents do not share their local training data with others.
After the algorithm performs $1\times {{10}^{5}}$ iterations, the accuracy for each digit is shown in Table I.
\begin{figure}[!htbp]
	\centering\resizebox{6cm}{6cm} {\includegraphics*{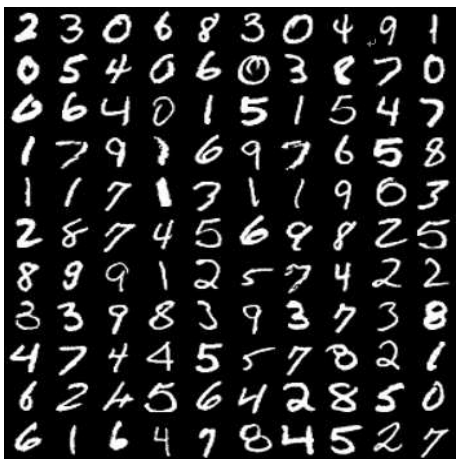}}
	\caption{Samples from the data set.}
\end{figure}

\begin{table}
\centering
\caption{Testing accuracy for classification.}
\begin{tabular}{lllllll}
\toprule
  Digit& 0 & 1 & 2 &3&4&\\
\midrule
  Accuracy &98.24\%  & 98.99\% & 96.91\% & 94.28\% & 97.16\%&\\
  \toprule
  Digit& 5 & 6 & 7 & 8&9&\\
  \midrule
 Accuracy  &95.47\%  & 97.18\% & 97.38\% & 92.87\% & 93.14\%&\\
\bottomrule
\end{tabular}
\end{table}
\subsection{Energy-Based Source Localization}
Consider a sensor network composed of $m$ agents distributed at known spatial locations, denoted ${{r}_{i}}\in {{\mathbb{R}}^{2}}$, $i=1,\ldots ,m$.
A stationary acoustic source is located at an unknown location ${{\tilde{x}}^{*}}\in {{\mathbb{R}}^{2}}$.
Let $a>0$ be a constant and ${{\upsilon }_{i,h}}$ be i.i.d. samples of a zero-mean Gaussian noise process with variance ${{\sigma }^{2}}$
We use an isotropic energy propagation model for the $h$-th received signal strength measurement at agent $i$:
${{c}_{i,h}}=\frac{a}{{{\left\| {\tilde{x}}-{{r}_{i}} \right\|}^{\theta }}}+{{v}_{i,h}}$ where $\left\| {\tilde{x}}-{{r}_{i}} \right\|>1$ and $\theta \ge 1 $ is an attenuation characteristic.
The maximum-likelihood estimator for the source's location is found by solving the problem
\begin{flalign}
{{\tilde{x}}^{*}}=\arg \underset{\tilde{x}\in {{\mathbb{R}}^{2}}}{\mathop{\min }}\,\sum\limits_{i=1}^{m}{\frac{1}{{{q}_{i}}}\sum\limits_{h=1}^{{{q}_{i}}}{{{\left( {{c}_{i,h}}-\frac{a}{{{\left\| \tilde{x}-{{r}_{i}} \right\|}^{\theta }}} \right)}^{2}}}}
\end{flalign}
The method that we use to solve this problem is proposed in \cite{Doron2006}.
According to the analysis given in \cite{Doron2006}, it can be found that the instantaneous function
\begin{flalign*}
f_{i}^{h}={{\left( {{c}_{i,h}}-\frac{a}{{{\left\| \tilde{x}-{{r}_{i}} \right\|}^{\theta }}} \right)}^{2}}
\end{flalign*}
obtains its minimum on the circle
\begin{flalign*}
{{C}_{i,h}}=\left\{ \tilde{x}\in {{\mathbb{R}}^{2}}:\left\| \tilde{x}-{{r}_{i}} \right\|=\sqrt{a/{{c}_{i,h}}} \right\}
\end{flalign*}
Let $D_{i,h}$ be the disk defined by
\begin{flalign*}
{{D}_{i,h}}=\left\{ \tilde{x}\in {{\mathbb{R}}^{2}}:\left\| \tilde{x}-{{r}_{i}} \right\|\le \sqrt{a/{{c}_{i,h}}} \right\}
\end{flalign*}
Then, the estimator is any point that minimizes the sum of squared distances to the sets ${{D}_{i,h}}$, $i=1,\ldots ,m$, $h=1,\ldots ,{{q}_{i}}$.
We rewrite the problem (52) as follows:
\begin{flalign}
{{{\tilde{x}}}^{*}}=\arg \underset{\tilde{x}\in {{\mathbb{R}}^{n}}}{\mathop{\min }}\,\sum\limits_{i=1}^{m}{\frac{1}{{{q}_{i}}}\sum\limits_{h=1}^{{{q}_{i}}}{{{\left\| \tilde{x}-{{\mathsf{\mathcal{P}}}_{{{D}_{i,h}}}}\left( {\tilde{x}} \right) \right\|}^{2}}}}
\end{flalign}
where ${{\mathsf{\mathcal{P}}}_{{{D}_{i,h}}}}\left( {\tilde{x}} \right)$ is the orthogonal projection of ${\tilde{x}}$ onto ${{D}_{i,h}}$.
\begin{figure}[!htbp]
	\centering\resizebox{8cm}{5.5cm} {\includegraphics*{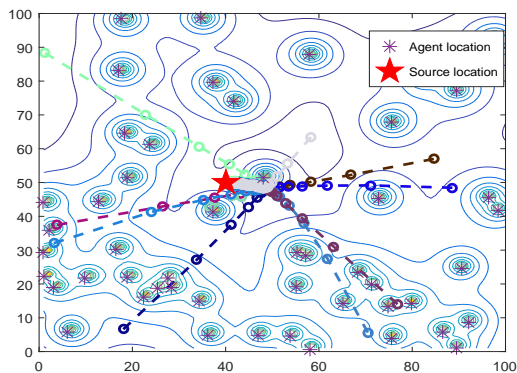}}
	\caption{A part of paths displayed on top of contours of log-likelihood function.}
\end{figure}
We have simulated this scenario with $50$ sensors uniformly distributed in a $100\times 100$ square, and the source location chosen randomly.
The source emits a signal with strength $a = 100$ and each sensor makes $100$ measurements.
Fig. 7 depicts $10$ paths taken by the S-DIGing algorithm plotted on top of contours of the log likelihood.
\subsection{Distributed K-Means Clustering}
A popular clustering method that minimizes the clustering error is the K-means algorithm \cite{likas2003global}.
Suppose that there is a data set $P=\sum\nolimits_{i=1}^{m}{\sum\nolimits_{h=1}^{{{q}_{i}}}{{{p}_{i,h}}}}$, where ${{p}_{i,h}}\in {{\mathbb{R}}^{n}}$.
The $K$-clustering aims at partitioning this data set into $K$ disjoint clusters ${{C}_{1}},\ldots ,{{C}_{K}}$, such that a
clustering criterion is optimized.
The most widely used clustering criterion is the sum of the squared Euclidean distances between each data point $p_{i,h}$ and the cluster center $m_{l}$ of the subset which contains $p_{i,h}$.
This criterion is called clustering error and depends on the cluster centers ${{m}_{1}},\ldots ,{{m}_{K}}$:
\begin{flalign}
\underset{\tilde{x}\in {{\mathbb{R}}^{Kn}}}{\mathop{\min }}\,f\left( {\tilde{x}} \right)=\sum\limits_{i=1}^{m}{\sum\limits_{h=1}^{{{q}_{i}}}{\sum\limits_{l=1}^{K}{a_{i,h}^{l}{{\left\| {{p}_{i,h}}-{{m}_{l}} \right\|}^{2}}}}}
\end{flalign}
where $\tilde{x}={{[m_{1}^{\text{T}},.\ldots ,m_{K}^{\text{T}}]}^{\text{T}}}$, $a_{i,h}^{l}=1$ if ${{p}_{i,h}}\in {{C}_{l}}$ and $a_{i,h}^{l}=0$ otherwise.
Then, we solve the clustering problem for the Iris data set provided in UCI Machine Learning Repository \cite{Dua2019}.
The data set contains 3 classes of 50 samples, where each class refers to a type of iris plant.
Each sample has 4 attributes included sepal length, sepal width, petal length and petal width.
We set $m=5$, ${{q}_{i}}=30$, $i=1,\ldots ,m$, and the probability of connection between each pair of agents is $40\%$.
Fig. 8 presents that although the S-DIGing algorithm needs more iterations than the DIGing algorithm to achieve a same residual.
The S-DIGing algorithm has an advantage in running time due to the smaller computational cost required for a single iteration.
\begin{figure}
\centering
\subfigure[Evolution of residuals with number of iterations]{
\includegraphics[width=8cm,height=5.5cm]{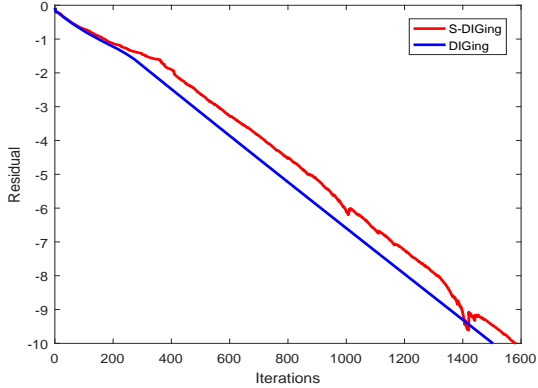}}
\subfigure[Evolution of residuals with running time of algorithms]{
\includegraphics[width=8cm,height=5.5cm]{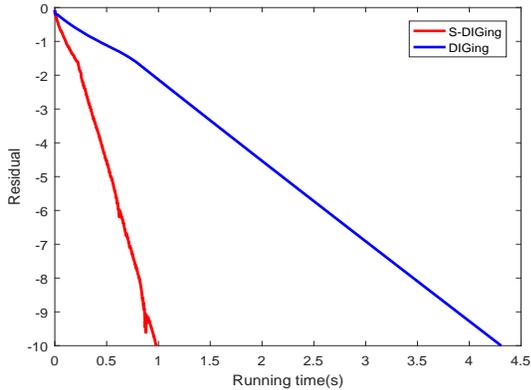}}
\caption{Comparison across S-DIGing and DIGing.}
\end{figure}
\section{Conclusion}
In this paper, a distributed stochastic gradient tracking algorithm which combines the gradient tracking algorithm with stochastic averaging gradient was proposed to solve the distributed optimization problem where each local objective function is constructed as an average of instantaneous functions.
Employing the unbiased stochastic gradient technology, the cost of calculating the gradient of local objective function at each agent is greatly reduced.
The theoretical analysis showed that the S-DIGing algorithm can linearly converge to the global optimal solution with explicit convergence rate when step-size is positive and less than an upper bound.
We presented three numerical simulations to illustrate the effectiveness of the S-DIGing algorithm.
Future work will focus on further improving the convergence rate and studying distributed optimization algorithm over time-varying and directed networks.
For good measure, different from the synchronous update and communication required by the existing algorithm, asynchronous distributed optimization algorithm is also a promising research.

\bibliographystyle{ieeetr}
\bibliography{library}

\end{document}